\documentclass[11pt]{amsart}
\usepackage[all]{xy}

\usepackage{amsmath}
\usepackage{amsfonts}
\usepackage{amssymb}
\usepackage{mathrsfs}

\DeclareFontEncoding{OT2}{}{} 
\newcommand{\textcyr}[1]{%
 {\fontencoding{OT2}\fontfamily{wncyr}\fontseries{m}\fontshape{n}
 \selectfont #1}}
\newcommand{\Sha}{{\!\be\lbe\mbox{\textcyr{Sh}}}}

\newcommand{\Bha}{{\!\be\lbe\mbox{\textcyr{B}}}}

\numberwithin{equation}{section}

\def\lra{\longrightarrow}
\def\kb{\overline{K}}
\def\kc{\kb\!\be\phantom{.}^{*}}
\def\pdiv{p\!-\!{\rm{div}}}
\def\ldiv{\ell\!-\!{\rm{div}}}
\def\bh{{\mathbb H}}
\def\bd{{\mathbb D}}
\def\ov{{\mathcal O}_{v}}

\def\ove{{\mathcal O}_{v,{\rm{\acute{e}t}}}}

\def\br{\text{Br}}
\def\bz{{\mathbb Z}\,}
\def\bq{{\mathbb Q}}
\def\bg{{\mathbb G}}
\def\spec{{\rm{Spec}}\,}

\def\img{{\rm{Im}}\,}

\def\der{\e\rm{der}}

\def\be{\kern -.1em}
\def\lbe{\kern -.025em}
\def\iv{{\rm{Div}}}
\def\pic{{\rm{Pic}}}
\def\hom{{\rm{Hom}\e}}

\def\krn{{\rm{Ker}}\e }
\def\img{{\rm{Im}}\e }
\def\cok{{\rm{Coker}}}

\def\tor{\e\rm{tor}}

\def\Gtil{{\widetilde{G}}}

\def\s{\mathscr }
\def\ra{\rightarrow}
\def\e{\kern 0.08em}
\def\le{\kern 0.03em}
\def\ng{\kern -0.04em}

\def\g{\varGamma}

\def\krn{{\rm{Ker}}\,}
\def\cok{{\rm{Coker}}\,}

\def\br{{\rm{Br}}\e}
\def\bra{{\rm{Br}}_{\rm{a}}\e}
\def\branr{{\rm{Br}}_{\rm{a,nr}}\e}
\def\branrs{{\rm{Br}}_{\rm{a,nr}}^{\e S}\e}
\def\xb{\overline{X}}
\def\gb{\overline{G}}

\newtheorem{lemma}{Lemma}[section]
\newtheorem{theorem}[lemma]{Theorem}
\newtheorem{corollary}[lemma]{Corollary}
\newtheorem{proposition}[lemma]{Proposition}
\theoremstyle{definition}
\newtheorem{definition}[lemma]{Definition}

\theoremstyle{remark}
\newtheorem{remark}[lemma]{Remark}
\newtheorem{remarks}[lemma]{Remarks}

\newtheorem{examples}[lemma]{Examples}

\begin{document}

\title[Abelian class groups of reductive group schemes]{Abelian class groups of
reductive group schemes}



\author{Cristian D. Gonz\'alez-Avil\'es}
\address{Departamento de Matem\'aticas, Universidad de La Serena, Chile}
\email{cgonzalez@userena.cl}


\thanks{The author is partially supported by Fondecyt grant
1080025}

\maketitle

\begin{abstract} We introduce the abelian class group $C_{\rm{ab}}(G)$
of a reductive group scheme $G$ over a ring $A$ of arithmetical
interest and study some of its basic properties. For example, we
show that if the fraction field of $A$ is a global field without
real primes, then there exists a surjection $C(G)\twoheadrightarrow
C_{\rm{ab}}(G)$, where $C(G)$ is the class set of $G$.
\end{abstract}

\section{Introduction}

Let $K$ a global field, i.e., $K$ is either a number field or a
function field in one variable over a finite field $k$. Let $S$ be a
nonempty open subscheme of the spectrum of the ring of integers of
$K$ (in the number field case) or a nonempty open {\it affine}
subscheme of the unique smooth, projective and irreducible curve
over $k$ whose function field is $K$ (in the function field case).
For any nonempty open subscheme $U$ of $S$, let $U_{0}$ denote the
set of closed points of $U$. For each $v\in S_{0}$, let $\ov$ denote
the completion of the local ring of $S$ at $v$ and let $K_{v}$
denote the fraction field of $\ov$. Let
$$
A_{S}(U)=\prod_{v\in S_{0}\setminus U_{0}} K_{v}\times \prod_{v\in
U_{0}} \ov
$$
be the ring of $U$-integral adeles of $S$. The rings $A_{S}(U)$ form
an inductive system when the sets $U$ are ordered by reverse
inclusion (i.e., $U\leq U^{\e\prime}$ if, and only if, $U\supset
U^{\e\prime}$) and the ring of adeles of $S$ is, by definition,
$$
A_{S}=\varinjlim_{U} A_{S}(U).
$$
Now let $G$ be an affine group scheme of finite type over $S$ with
smooth generic fiber. Then the class set of $G$,
$$
C(G)=G\be\left(\lbe{\mathbb A}_{\e S}(S)\lbe\right)\backslash\e
G\be\left(\lbe{\mathbb A}_{\e S}\lbe\right)\be/\e G(\lbe K\lbe),
$$
encodes important arithmetic information about $G$. Unfortunately,
the above set is notoriously difficult to compute and, in general,
carries no additional structure, which makes the tools of Algebra
rather useless in its study. It is perhaps for this reason that, at
least in specific cases, researchers have sought to study certain
abelian groups which in some sense ``approximate" $C(G)$. A
well-known instance of this approach can be found in the theory of
quadratic forms, where the genus class set of a quadratic form $q$
(or, equivalently, the class set of a natural integral model of its
orthogonal group) is studied via an ``abelian group approximation",
namely the spinor genus class group of $q$. In this paper we present
the beginnings of a general development of this idea. More
precisely, we introduce the {\it abelian class group} of a reductive
$S$-group scheme $G$ and study some of its main properties (see
below for statements). Thus the theory developed here applies to
(connected) reductive algebraic groups over $K$ which extend to a
reductive group scheme over $S$, i.e., which have {\it good
reduction over $S$}. We work under this restriction because the
general case, where $G$ is allowed to have ``bad" (non-reductive)
fibers, requires a long series of preparations\footnote{ Including
the development of a homotopy theory for smooth (abelian) sheaves
over a discrete valuation ring and a theory of N\'eron models of
(2-term) complexes of tori.} which will be the subject of separate
papers. In this sense, therefore, the present paper is less general
than the work of C.Demarche \cite{dem2}, who considered arbitrary
flat integral models of finite type of a (connected) reductive
algebraic group $G_{\be K}$ over a number field $K$. See
\cite{dem2}, \S 4. However, in a different sense to be explained
below, the present paper is more general than [op.cit.]. Indeed, the
arguments of [op.cit.] are valid only under the hypothesis (H) that
the simply-connected central cover $\Gtil_{\lbe K}$ of $G_{\be K}$
satisfies the strong approximation property with respect to the set
$\Sigma$ of primes of $K$ which do not correspond to a point of
$S_{0}$. Under hypothesis (H), the class sets of the models
considered in \cite{dem2} are naturally equipped with the structure
of (finite) abelian groups, and Demarche is able to use the
Brauer-Manin pairing to obtain a duality theorem for these groups
\cite{dem2}, Theorem 4.1. We stress here the evident fact that,
since the Brauer group of a scheme is abelian, the Brauer-Manin
pairing can only ``detect" {\it abelian} groups associated to $G$.
In this regard, we also note that (H) implies as well that the
defect of strong approximation for $G_{\be K}$ relative to $\Sigma$
is naturally equipped with the structure of an abelian group, and
\cite{dem2}, Theorem 3.14, can be interpreted as a duality statement
for this group. In this paper we dispense with the hypothesis (H) in
the study of class sets of reductive group schemes over $S$ and, in
particular, abandon the compact-noncompact-type dichotomy which is
familiar from the discussion of class sets contained in \cite{pr},
\S\S8.2 and 8.3. We are thus able to handle both cases
simultaneously (again, under a ``good reduction" hypothesis on
$G_{\be K}$). Note, further, that this paper also covers the
function field case.

\smallskip

We now state the main results of the paper.

\smallskip

Let $G$ be a reductive group scheme over $S$, let $G^{\der}$ be the
derived group of $G$ and set $G^{\tor}=G/G^{\der}$. Further, let
$\Gtil$ be the simply-connected central cover of $G^{\der}$ and let
$\mu$ for the fundamental group of $G^{\der}$, i.e., the kernel of
$\Gtil\ra G^{\der}$. Now let $S_{\rm{fl}}$ be the small fppf site
over $S$ and let $H^{i}_{\rm{ab}}(S_{\rm{fl}},G)$ be the abelian
cohomology groups of $G$ introduced in \cite{ga3}. Restriction to
the generic fiber of $G$ yields a map
$H^{1}_{\rm{ab}}(S_{\rm{fl}},G)\ra H^{1}_{\rm{ab}}(K_{\rm{fl}},G)$,
and the {\it abelian class group of $\e G$} is by definition
$$
C_{\rm{ab}}(G)=\krn\!\be\left[\e H^{1}_{\rm{ab}}(S_{\rm{fl}},G)\ra
H^{1}_{\rm{ab}}(K_{\rm{fl}},G)\e\right].
$$
When $K$ has no real primes and $\Gtil_{\be K}$ satisfies hypothesis
(H), so that $C(G)$ has a natural abelian group structure as
mentioned above, then there exists a canonical isomorphism
$C(G)\simeq C_{\rm{ab}}(G\e)$ (see Remark 3.12(b))\footnote{When $K$
is a number field with real primes, the relation between $C(G)$ and
$C_{\rm{ab}}(G\e)$ is more complicated. See Remark 3.12(c).}. Then
results from Section 3, which essentially follow from the main
theorem of \cite{ga3}, yield the following statement.

\begin{theorem} {\rm{(=Theorem 3.11)}} Assume that $K$ has no real primes.
Then there exist a natural right action of $H^{\e
0}_{\rm{ab}}(S_{\rm{fl}},G \e)$ on $H^{\e
1}\be\big(S_{\e\rm{\acute{e}t}},\Gtil\e\big)$ and a canonical exact
sequence of pointed sets
$$
1\ra H^{\e 1}\be\big(S_{\rm{\acute{e}t}},\Gtil\e\big)/H^{\e
0}_{\rm{ab}}(S_{\rm{fl}},G \e)\ra C(G\e)\ra C_{\rm{ab}}(G\e)\ra 1,
$$
where the first nontrivial map is injective.
\end{theorem}

In Section 3 we also study some basic properties of
$C_{\rm{ab}}(G\e)$. In particular, using the flasque resolutions of
$G$ constructed in \cite{ga4}, we obtain the following results.

\begin{theorem} {\rm{(=Theorem 3.13)}} Let $G$ be a reductive group scheme over $S$
and let $1\ra F\ra H \ra G \ra 1$ be a flasque resolution of $G$. Set
$R=H^{\lbe\tor}$. Then the given resolution induces an exact
sequence of finitely generated abelian groups
$$
\mu(S)\hookrightarrow F(S)\ra R(S)\ra
H^{0}_{\rm{ab}}(S_{\rm{fl}},G\le)\ra H^{\le
1}(S_{\e\rm{\acute{e}t}},F\e)\ra C(R)\ra C_{\rm{ab}}(G)\ra 1.
$$
\end{theorem}

\smallskip

\begin{theorem} {\rm{(=Corollary 3.14)}} Let $L/K$ be a finite Galois extension and
let $S^{\e\prime}\ra S$ be the normalization of $S$ in $L$. Let $G$
be a reductive group scheme over $S$ and let $1\ra F\ra H \ra G \ra
1$ be a flasque resolution of $G$. Then the given resolution defines
a {\rm corestriction homomorphism}
$$
{\rm{cores}}_{ S^{\prime}\be/S}\colon
C_{\rm{ab}}(S^{\e\prime},G\e)\ra C_{\rm{ab}}(S,G\e),
$$
where $C_{\rm{ab}}(S^{\e\prime},G\e)$ (respectively,
$C_{\rm{ab}}(S,G\e)$) is the abelian class group of
$G\times_{S}S^{\e\prime}$ (respectively, $G$).
\end{theorem}

The homomorphism of the theorem is independent, up to isomorphism,
of the chosen flasque resolution of $G$ and is functorial in
$S^{\e\prime}\ra S$. See Remark 3.15. When $K$ has no real primes
and $\Gtil_{\be K}$ satisfies hypothesis (H), so that $C(G)\simeq
C_{\rm{ab}}(G\e)$ as noted above, the preceding theorem shows that
$C(G)$ is endowed with natural corestriction maps. When $K$ is an
arbitrary number field and $\Gtil_{\be K}$ satisfies hypothesis (H),
C.Demarche obtained in \cite{dem2}, Theorem 4.6, a similar
corestriction homomorphism which presumably coincides with the above
one when $K$ is totally imaginary. The existence of such
corestriction maps (under hypothesis (H) and for any $K$) was first
established in \cite{th}, Theorem 14, p.36.

In Sections 4 and 5 we use results of C.Demarche \cite{dem1} and
M.Borovoi and J.van Hamel \cite{bvh} to obtain the following result.
Set $G_{\be K}=G\times_{S}\spec K$.

\begin{theorem} {\rm{(=Theorem 5.5)}} Assume that $G_{\be K}=G\times_{S}\spec K$ admits
a smooth $K$-compactification\footnote{This is certainly the case if
$K$ is a number field, by Hironaka's theorem.}. There exists a
perfect pairing of finite groups
$$
C_{\rm{ab}}(G\e)\times \branrs(G_{\be K}\e)/\Bha(G_{\be
K}\e)\ra\bq/\bz,
$$
where $\branrs(G_{\be K}\e)$ and $\Bha(G_{\be K}\e)$ are the
subgroups of the algebraic Brauer group of $G_{\be K}$ given by
\eqref{branrs} and \eqref{bg}, respectively.
\end{theorem}

When $K$ is a totally imaginary number field and $\Gtil_{\be K}$
satisfies hypothesis (H), the pairing of the theorem should be
closely related to that obtained by Demarche in \cite{dem2}, Theorem
4.14. We hope to clarify this issue in a future publication.

\section*{Acknowledgements}
I thank Bas Edixhoven, Philippe Gille, Niko Naumann and Adrian Vasiu
for helpful comments.

\section{Preliminaries}

Let $K$ and $S$ be as in the Introduction. We will write
$S_{\e\rm{fl}}$ (respectively, $S_{\e\rm{\acute{e}t}}$,
$S_{\rm{Nis}}$) for the small fppf (respectively, \'etale,
Nisnevich) site over $S$. If $\tau=\rm{fl}$ or $\rm{\acute{e}t}$,
$G$ is an $S$-group scheme and $i=0$ or $1$, $H^{\le
i}(S_{\e\tau},G)$ will denote the $i$-th cohomology set of the sheaf
on $S_{\e\tau}$ represented by $G$. If $G$ is commutative, these
cohomology sets are in fact abelian groups and are defined for every
$i\geq 0$. If $S^{\e\prime}\ra S$ is a morphism of schemes, we will
write $H^{\le i}(S_{\e\rm{fl}}^{\e\prime},G)$ for $H^{\le
i}(S_{\e\rm{fl}}^{\e\prime},G\times_{S}S^{\e\prime}\e)$. When $G$ is
smooth, the canonical map $H^{\le i}(S_{\rm{\acute{e}t}},G)\ra H^{\e
i}(S_{\e\rm{fl}},G)$ is bijective (see \cite{mi1}, Remark
III.4.8(a), p.123). In this case, the preceding sets will be
identified. If $K$ is a field, we will write $\kb$ for a fixed {\it
separable} algebraic closure of $K$, $\g$ for $\text{Gal}(\kb/K)$
and $H^{\le i}(K,G)$ for the Galois cohomology set (or group) $H^{\e
i}\big(\g,G\big(\e\kb\e\big)\big)$. If $S=\spec K$, $H^{\le
i}(S_{\rm{\acute{e}t}},G)$ and $H^{\le i}(K,G)$ will be identified.

An $S$-group scheme $G$ is called {\it reductive} (respectively,
{\it semisimple}) if it is affine and smooth over $S$ and its
geometric fibers are  {\it connected} reductive (respectively,
semisimple) algebraic groups. If $G$ is a reductive $S$-group
scheme, $G^{\le *}=\hom_{\e S-\text{gr}}(G,\bg_{m,S})$ is the
twisted-constant $S$-group scheme of characters of $G$. The derived
group of $G$ (see \cite{sga3}, XXII, Theorem 6.2.1(iv)), will be
denoted by $G^{\der}$. It is a normal semisimple subgroup scheme of
$G$ and the quotient
$$
G^{\tor}:=G/G^{\der}
$$
is an $S$-torus (in \cite{sga3}, XXII, 6.2, $G^{\tor}$ is denoted by
$\rm{corad}(G\e)$ and called the {\it coradical} of $G$). Note that,
since $G^{\der\le*}=0$, the exact sequence of reductive $S$-group
schemes
\begin{equation}\label{seq1}
1\ra G^{\der}\ra G\ra G^{\tor}\ra 1
\end{equation}
induces the equality $G^{\le *}=G^{\tor\le*}$.

A semisimple $S$-group scheme $G$ is called {\it simply-connected}
if it admits no nontrivial central cover, i.e., any central
$S$-isogeny $G^{\e\prime}\ra G$ from a semisimple $S$-group scheme
$G^{\e\prime}$ to $G$ is an isomorphism. If $G$ is any semisimple
$S$-group scheme, then there exists a simply-connected $S$-group
scheme $\Gtil$ and a central isogeny $\varphi\colon\Gtil\ra G$. The
pair $(\Gtil,\varphi)$ is unique up to unique isomorphism and its
formation commutes with arbitrary extensions of the base. It is
called the {\it simply-connected central cover of $\,G$}. The {\it
fundamental group} of $G$ is by definition the kernel of $\varphi$
and will be denoted by $\mu$ (or by $\mu_{G}$, if necessary). It is
a finite $S$-group scheme of multiplicative type. Now, if $G$ is any
reductive $S$-group scheme, $\Gtil$ will denote the simply-connected
central cover of its derived group $G^{\der}$ (which is semisimple).
Further, the fundamental group of $G$ is by definition that of
$G^{\der}$ and will be denoted by $\mu$ (or by $\mu_{G}$). There
exists a canonical central extension
\begin{equation}\label{seq2}
1\ra\mu\ra\Gtil\ra G^{\der}\ra 1.
\end{equation}
We will write $\partial\colon\Gtil\ra G$ for the composition
$\Gtil\twoheadrightarrow G^{\der}\hookrightarrow G$. Clearly, there
exists a canonical exact sequence
\begin{equation}\label{seq3}
1\ra\mu\ra\Gtil\overset{\partial}\ra G\ra G^{\tor}\ra 1.
\end{equation}
Now there exists a canonical ``conjugation" action of $G$ on $\Gtil$
such that $(\Gtil\be\overset{\partial}\ra\be G\e)$, regarded as a
two-term complex with $\Gtil$ and $G$ placed in degrees $-1$ and
$0$, respectively, is a (left) {\it quasi-abelian} crossed module on
$S_{\e\rm{fl}}$, in the sense of \cite{ga3}, Definition 3.2. See
\cite{br}, Example 1.9, p.28, and \cite{ga3}, Example 2.2(iii). Thus
$\partial$ induces a homomorphism $\partial_{Z}\colon
Z\big(\Gtil\,\big)\ra Z(G)$ and the embedding of crossed modules
$$
\big(Z\big(\Gtil\e\big)\!\overset{\be\partial_{\lbe Z}}
\longrightarrow\! Z(G)\big)\hookrightarrow
(\Gtil\!\overset{\partial}\longrightarrow\! G)
$$
is a quasi-isomorphism (see \cite{ga3}, Proposition 3.4). In
particular, \eqref{seq3} induces an exact sequence of $S$-groups of
multiplicative type
\begin{equation}\label{seq4}
1\ra\mu\ra Z\big(\Gtil\e\big)\overset{\partial_{Z}}\longrightarrow
Z(G) \ra G^{\tor}\ra 1.
\end{equation}

Let $i\geq -1$ be an integer. The $i$-th \textit{abelian (flat)
cohomology group of $G$} is by definition the hypercohomology group
\begin{equation}\label{abcohom}
H^{\le i}_{\rm{ab}}(S_{\e\rm{fl}}, G\e)={\bh}^{\e
i}\big(S_{\e\rm{fl}},
Z\big(\Gtil\,\big)\!\overset{\partial_{Z}}\longrightarrow\! Z(G\e)),
\end{equation}
where $Z(G)$ is placed in degree 0. We will also need the
\textit{dual abelian cohomology groups of $\e G$}. By definition,
these are the groups
\begin{equation}\label{dabcohom}
H^{\le i}_{\rm{ab}}(S_{\e\rm{\acute{e}t}},G^{*})={\bh}^{\e
i}\big(S_{\e\rm{\acute{e}t}},Z(G\e)^{*}\!\overset{\partial_{\be
Z}^{\e *}}\longrightarrow Z\big(\Gtil\,\big)^{\lbe *}\e\big),
\end{equation}
where $Z(G\e)^{*}$ is placed in degree $-1$. Note that, as the
Cartier dual of an $S$-group of multiplicative type is \'etale over
$S$, the preceding groups coincide with the flat hypercohomology
groups ${\bh}^{\e i}\big(S_{\e\rm{fl}},Z(G\e)^{*}\ra
Z\big(\Gtil\,\big)^{\lbe *}\e\big)$. If $S=\spec K$, where $K$ is a
field, we will write $H^{\le i}_{\rm{ab}}(K,G^{*})$ for $H^{\le
i}_{\rm{ab}}(S_{\e\rm{\acute{e}t}},G^{*})$. Clearly,
$$
H^{\le i}_{\rm{ab}}(K,G^{*})={\bh}^{\e
i}\big(\g,Z\big(G\big(\e\kb\e\big)\big)^{\be *}\ra
Z\big(\Gtil\big(\e\kb\e\big)\big)^{\be *}\e\big)
$$
(Galois hypercohomology). Now, by \eqref{seq4}, there exist exact
sequences
\begin{equation}\label{kamb}
\dots\ra H^{\e i-1}(S_{\e\rm{\acute{e}t}},G^{\tor}\e)\ra H^{\le
i+1}(S_{\e\rm{fl}},\mu\e)\ra H^{\e i}_{\rm{ab}}(S_{\e\rm{fl}},G)\ra
H^{\e i}(S_{\e\rm{\acute{e}t}},G^{\tor}\e)\ra\dots.
\end{equation}
and
\begin{equation}\label{kamb*}
\dots\ra H^{\e i-1}(S_{\e\rm{\acute{e}t}},\mu^{\lbe *}\e)\ra H^{\le
i+1}(S_{\e\rm{\acute{e}t}},G^{\tor*}\e)\ra H^{\e
i}_{\rm{ab}}(S_{\e\rm{\acute{e}t}},G^{*})\ra H^{\e
i}(S_{\e\rm{\acute{e}t}},\mu^{\lbe *}\e)\ra\dots.
\end{equation}

\begin{examples}\indent

\begin{enumerate}

\item[(a)] If $G$ is semisimple, i.e., $G^{\tor}=0$, then
$H^{\e i}_{\rm{ab}}(S_{\e\rm{fl}},G)=H^{\le
i+1}(S_{\e\rm{fl}},\mu\e)$ and $H^{\e
i}_{\rm{ab}}(S_{\e\rm{\acute{e}t}},G^{*})=H^{\e
i}(S_{\e\rm{\acute{e}t}},\mu^{\lbe *})$.
\item[(b)] If $G$ has trivial fundamental group, i.e.,
$\mu=0$, then $H^{\e i}_{\rm{ab}}(S_{\e\rm{fl}},G)=H^{\e
i}(S_{\e\rm{\acute{e}t}},G^{\tor}\e)$ and $H^{\e
i}_{\rm{ab}}(S_{\e\rm{\acute{e}t}},G^{*})=H^{\le
i+1}(S_{\e\rm{\acute{e}t}},G^{\tor*}\e)$.
\end{enumerate}

\end{examples}

By \cite{ga3}, there exist canonical {\it abelianization} maps
\begin{equation}\label{abmaps}
{\rm{ab}}^{i}={\rm{ab}}^{i}_{G\lbe/\lbe S}\colon H^{\le
i}(S_{\e\rm{fl}},G)\ra H^{\e i}_{\rm{ab}}(S_{\e\rm{fl}},G\e)
\end{equation}
so that the following holds.

\begin{theorem} Let $G$ be a reductive group scheme over $S$.
\begin{enumerate}
\item[(i)] There exists an exact sequence of pointed sets
$$
\begin{array}{rcl}
1&\longrightarrow&\mu(S)\ra\Gtil(S)\ra G (S)
\overset{\rm{ab}^{0}}\longrightarrow H^{\e
0}_{\rm{ab}}(S_{\e\rm{fl}},G \e)\overset{\delta_{\le 0}}\lra H^{\le
1}\big(S_{\e\rm{fl}},\Gtil\e\big)\\
&\overset{\,\partial^{\e(1)}}\lra & H^{\le
1}(S_{\e\rm{fl}},G\e)\overset{\rm{ab}^{1}}\longrightarrow H^{\e
1}_{\lbe\rm{ab}}(S_{\e\rm{fl}},G \e)\ra 1.
\end{array}
$$
\item[(ii)] The group $H^{\e 0}_{\rm{ab}}(S_{\e\rm{fl}},G \e)$
acts on the right on the set $H^{\le
1}\big(S_{\e\rm{fl}},\Gtil\e\big)$ compatibly with the map
$\delta_{\le 0}$ and the preceding exact sequence induces an exact
sequence of pointed sets
$$
1\ra H^{\le 1}\big(S_{\e\rm{fl}},\Gtil\e\big)/H^{\e
0}_{\rm{ab}}(S_{\e\rm{fl}},G \e)\overset{\bar{\partial}^{\e(1)}}\lra
H^{\le 1}(S_{\e\rm{fl}},G\e)\overset{\rm{ab}^{1}}\longrightarrow
H^{\e 1}_{\lbe\rm{ab}}(S_{\e\rm{fl}},G\e)\ra 1,
$$
where the map $\bar{\partial}^{\e(1)}$, induced by
$\partial^{\e(1)}$, is {\rm injective}.
\end{enumerate}
\end{theorem}
\begin{proof} By \cite{ga3}, Example 5.4(iii), $S$ is a scheme
{\it of Douai type}, i.e., every class of the Giraud cohomology set
$H^{\le 2}\big(S_{\e\rm{fl}},\Gtil\e\big)$ is neutral. Thus the
theorem follows from \cite{ga3}, Theorem 5.1, Theorem 5.5 and
Proposition 3.14(b). The action mentioned in (ii) is defined in
[op.cit.], Remark 3.9(b).
\end{proof}

\begin{remark} The exact sequence in part (ii) of the theorem is
compatible with inverse images, i.e., if $S^{\e\prime}\ra S$ is a
morphism of schemes of Douai type, then the following diagram
commutes
$$
\xymatrix{1\ra H^{\e
1}\be\big(S_{\e\rm{\acute{e}t}},\Gtil\e\big)/H^{\e
0}_{\rm{ab}}(S_{\e\rm{fl}},G\e)\ar[r]\ar[d]& H^{\e
1}\be(S_{\e\rm{\acute{e}t}},G\e)\ar[r]\ar[d]& H^{\e
1}_{\rm{ab}}(S_{\e\rm{fl}},G\e)\ar[d]\ar[r]&1\\
1\ra H^{\e
1}\be\big(S^{\e\prime}_{\rm{\acute{e}t}},\Gtil\e\big)/H^{\e
0}_{\rm{ab}}(S^{\e\prime}_{\rm{fl}},G\e)\ar[r]& H^{\e
1}\be(S^{\e\prime}_{\rm{\acute{e}t}},G\e)\ar[r]& H^{\e
1}_{\rm{ab}}(S^{\e\prime}_{\rm{fl}},G\e)\ar[r]&1.}
$$
This follows from \cite{ga3}, Remark 4.3, and the fact that the
action of $H^{\e 0}_{\rm{ab}}(S_{\e\rm{fl}},G\e)$ on $H^{\e
1}\be\big(S_{\e\rm{\acute{e}t}},\Gtil\e\big)$ is compatible with
inverse images (see [op.cit.], Remark 3.9(b)).
\end{remark}

\medskip

We will write $S_{0}$ for the set of closed points of $S$ and
$\Sigma$ for the set of primes of $K$ which do not correspond to a
point of $S_{0}$. Thus $\Sigma$ is nonempty and contains all
archimedean primes of $K$ in the number field case. For any prime
$v$ of $K$, $K_{\lbe v}$ will denote the completion of $K$ at $v$.
If $v\in S_{0}$, we will write $\ov$ for the ring of integers of
$K_{v}$ and $k(v)$ for the corresponding residue field. Note that,
since $k(v)$ is finite, every $k(v)$-torus is cohomologically
trivial by \cite{se}, Propositions 5(iii) and 6(b), pp. II-7-8. If
$v$ is a real prime of $K$, $i$ is any integer and $C$ is a
cohomologically bounded complex of abelian sheaves on $(\spec
K_{v})_{\rm{fl}}$, ${\bh}^{\e i}\big(K_{v,\rm{fl}},C)$ will denote
the {\it modified} (Tate) $i$-th hypercomology group of $C$ defined
in \cite{hs}, p.103. In particular, the groups $H^{i}_{\rm{ab}}
(K_{v,\rm{fl}},G)={\bh}^{\e i}\big(K_{v},Z\big(\lbe\Gtil\big)\ra
Z(G\e))$ coincide with the groups denoted $H^{i}_{\rm{ab}}(K_{v},G)$
in \cite{bor}.

\section{The abelian class group}

Let $K$ and $S$ be as in the Introduction. Let $G$ be a (connected)
reductive algebraic group over $K$ and let $T$ be a $K$-torus. Set
$$\begin{array}{rcl}
\Sha^{1}\lbe(K,G)&=&\krn\!\!\left[\e H^{1}(K,G\e)
\ra\displaystyle\prod_{\text{all $v$}}H^{1}(K_{v},G\e)\right],\\\\
\Sha^{\le 2}\lbe(K,T)&=&\krn\!\!\left[\e H^{\e
2}(K,T\e)\ra\displaystyle \prod_{\text{all $v$}}H^{\e
2}(K_{v},T\e)\right]
\end{array}
$$
and
\begin{equation}\label{absha}
\Sha^{1}_{\rm{ab}}\lbe(K,G)=\krn\!\!\left[\e
H^{1}_{\rm{ab}}(K_{\e\rm{fl}},G\e) \ra\displaystyle \prod_{\text{all
$v$}}H^{1}_{\rm{ab}} (K_{v,\rm{fl}},G\e)\right].
\end{equation}

\begin{proposition} Let $G$ be a (connected) reductive algebraic group over $K$. Let $1\ra F\ra H\ra G\ra 1$ be a
flasque resolution of $G$. Then the given resolution defines an
isomorphism
$$
\Sha^{1}_{\rm{ab}}\lbe(K,G)\simeq\Sha^{\le 2}(K,F).
$$
\end{proposition}
\begin{proof} Since $\Sha^{\le 2}(L,\bg_{m})=0$ for every finite separable
extension $L$ of $K$, $\Sha^{\le 2}(K,R)=0$. The result is now
immediate from \cite{ga4}, Proposition 4.5.
\end{proof}

\begin{remark} By \cite{ga3}, Corollary 5.10, there exists a bijection
$\Sha^{1}\lbe(K,G\e)\simeq\Sha^{1}_{\rm{ab}}\lbe(K,G\e)$. Thus the
proposition yields a bijection $\Sha^{1}\lbe(K,G\e)\simeq\Sha^{\le
2}(K,F)$ which extends to an arbitrary global field $K$ the
bijection of \cite{ct}, Theorem 9.4(ii). Now, if $G$ is {\it
semisimple}, then \cite{ga4}, (3.4) (with $G^{\tor}=0$), yields an
isomorphism $\Sha^{2}(K,\mu)\simeq\Sha^{\le 2}(K,F)$. Thus there
exists a bijection $\Sha^{1}\lbe(K,G\e)\simeq\Sha^{2}(K,\mu)$ which
extends to the function field case the bijection of \cite{san},
Corollary 4.4.
\end{remark}

Now let
\begin{equation}\label{shas}
\Sha^{\le 1}_{S}(K,G\e)=\krn\!\!\left[\e H^{\e
1}\lbe(K,G\e)\ra\displaystyle \prod_{v\in S_{0}}H^{\e
1}\be\big(K_{v},G\e\big)\right]
\end{equation}
and
\begin{equation}\label{shabs}
\Sha^{1}_{{\rm{ab}},S}\e(K,G\e)=\krn\!\!\left[\e H^{1}_{\rm{ab}}(K_{
\rm{fl}},G\e)\ra\displaystyle \prod_{v\in
S_{0}}H^{1}_{\rm{ab}}\be\big(K_{v,{\rm{fl}}},G\e\big)\right].
\end{equation}
The latter group is torsion since, by \eqref{kamb},
$H^{1}_{\rm{ab}}(K_{ \rm{fl}},G\e)$ is a torsion group.

\begin{lemma} The abelianization map $\e{\rm{ab}}_{G\lbe/\lbe K}^{1}\colon H^{\e 1}
\lbe(K,G\e)\ra H^{1}_{\rm{ab}}(K_{ \rm{fl}},G\e)$ induces a
surjection
$$
\Sha^{\le 1}_{S}(K,G\e)\twoheadrightarrow \Sha^{1}_{{\rm{ab}},
S}\e(K,G\e).
$$
\end{lemma}
\begin{proof} This follows from the commutative diagram
$$
\xymatrix{H^{\le 1}(K,G\e)\ar@{->>}[r]^(.45){{\rm{ab}}_{G\lbe/\lbe
K}^{1}} \ar[d]& H^{\le
1}_{\rm{ab}}(K_{\e\rm{fl}},G\e)\ar[d] \\
\displaystyle\prod_{v\in S_{0}}H^{1}(K_{v},G\e)\ar[r]^(.45){\sim}&
\displaystyle\prod_{v\in S_{0}}H^{\le 1}_{\rm{ab}}(K_{\e
v,\rm{fl}},G\e),}
$$
whose top map is surjective by \cite{ga3}, Theorem 5.5(i), and
bottom map is bijective by [op.cit.], Theorem 5.8(i) and Remark
5.9(a).
\end{proof}

Now let $G$ be any affine $S$-group scheme of finite type with
smooth generic fiber. Let $\mathbb A_{S}$ (respectively, $\mathbb
A_{S}(S)$) denote the ring of adeles (respectively, integral adeles)
of $S$ and let
$$
C(G\e)=G\be\left(\lbe{\mathbb A}_{\e S}(S)\lbe\right)\backslash\e
G\be\left(\lbe{\mathbb A}_{\e S}\lbe\right)\be/\e G(\lbe K\lbe)
$$
be the class set of $G$. Since $G$ is generically smooth,
Ye.Nisnevich has shown that there exists a canonical bijection
$C(G\e)\simeq H^{\le 1}(S_{\rm{Nis}},u_{*}G\e)$, where $u\colon
S_{\e\rm{\acute{e}t}}\ra S_{\rm{Nis}}$ is the canonical morphism of
sites. See \cite{ga2}, Theorem 3.5.

\begin{theorem} {\rm(Ye.Nisnevich)} Let $G$ be an affine group scheme of
finite type over $S$ with smooth generic fiber. Assume that $H^{\e
1}\lbe(\ove,G\e)$ is trivial for every $v\in S_{0}$. Then there
exists an exact sequence of pointed sets
$$
1\ra C(G\e)\overset{\varphi}\lra H^{\e
1}\lbe(S_{\e\rm{\acute{e}t}},G\e)\ra\Sha^{\le 1}_{S}(K,G\e)\ra 1,
$$
where the map $\varphi$ is injective and $\Sha^{\le 1}_{S}(K,G\e)$
is the set \eqref{shas}.
\end{theorem}
\begin{proof} This follows from
\cite{gi}, Proposition V.3.1.3, p.323, using the bijection
$C(G\e)\simeq H^{\le 1}(S_{\rm{Nis}},u_{*}G\e)$ recalled above and
\cite{nis}, Proposition 1.37, p.282.
\end{proof}

\begin{remark} The fibers of the map $\lambda\colon H^{\e
1}\lbe(S_{\e\rm{\acute{e}t}},G\e)\ra\Sha^{\le 1}_{S}(K,G\e)$
appearing above may be computed as follows (see \cite{gi}, Corollary
V.3.1.4, p.324): let $p\in H^{\e 1}\lbe(S_{\e\rm{\acute{e}t}},G\e)$,
choose a $G$-torsor $P$ representing $p$ and let
$\!\phantom{.}^{P}\be\lbe G$ be the twist of $G$ by $P$ (see
\cite{gi}, Proposition III.2.3.7, p.146). Let $\theta_{P}\colon
H^{\e 1}\be(S_{\e\rm{\acute{e}t}},G\e)\overset{\!\sim}\ra H^{\e
1}\be(S_{\e\rm{\acute{e}t}},\be\!\phantom{.}^{P}\be G\e)$ be the
bijection defined in [op.cit.], Remark III.2.6.3, p.154. Then
$\theta_{P}^{-1}\circ\be\!\phantom{.}^{P}\!\varphi$ induces a
bijection
$$
C\big(\be\!\phantom{.}^{P}\! G\e\big)\overset{\!\sim}\ra\{\e
p^{\e\prime}\in H^{\e 1}\be(S_{\e\rm{\acute{e}t}},G\e)\colon
\lambda\be\left(p^{\e\prime}\e\right)=\lambda(p)\}.
$$
\end{remark}

\smallskip

\begin{corollary} Let $G$ be a reductive group scheme over $S$.
Then there exists an exact sequence of pointed sets
$$
1\ra C(G\e)\ra H^{\e 1}\lbe(S_{\e\rm{\acute{e}t}},G\e)\ra\Sha^{\le
1}_{S}(K,G\e)\ra 1.
$$
\end{corollary}
\begin{proof} By the theorem, we only need to check that
$H^{\e 1}\lbe(\ove,G\e)$ is trivial for every $v\in S_{0}$. By
\cite{sga3}, XXIV, Proposition 8.1, there exists a canonical
bijection $H^{1}(\ove,G\e)\simeq H^{1}(k(v),G\e)$. Now, since
$G\times_{S}\spec k(v)$ is a connected algebraic group over the
finite field $k(v)$, the set $H^{1}(k(v),G\e)$ is trivial by Lang's
theorem \cite{la}, Theorem 2, p.557.
\end{proof}

\begin{corollary} Let $G$ be a semisimple and simply-connected group scheme over $S$.
\begin{enumerate}
\item[(i)] There exists an exact sequence of pointed sets
$$
1\ra C(G\e)\ra H^{\e 1}(S_{\e\rm{\acute{e}t}},G\e) \ra H^{\e
1}(K,G\e)\ra 1.
$$
\item[(ii)] The canonical localization map $H^{\e
1}(K,G\e)\ra\prod_{\,v\,\rm{real}}H^{\le 1}(K_{v},G\e)$ is
bijective.
\end{enumerate}
\end{corollary}
\begin{proof} Since $H^{\e 1}(K_{v},G\e)$
is trivial for every nonarchimedean prime $v$ of $K$ (see \cite{pr},
Theorem 6.4, and \cite{bt}, Theorem 4.7(ii)), we have $\Sha^{\le
1}_{S}(K,G\e)=H^{\le 1}(K,G\e)$. Thus assertion (i) is immediate
from the previous corollary. Assertion (ii) is well-known (see
\cite{pr}, Theorem 6.6, p.286, and \cite{har}, Theorem A, p.125).
\end{proof}

Henceforth, $G$ will denote a reductive group scheme over $S$. By
\cite{ga3}, Remark 4.3, there exists a commutative diagram
$$
\xymatrix{H^{\le 1}(S_{\e\rm{\acute{e}t}},G\e)\ar@{->>}[r]
\ar@{->>}[d]^(0.45){{\rm{ab}}^{1}_{\lbe G\lbe/\be S}}& \Sha^{\le
1}_{S}(K,G\e)\ar[d] \\
H^{1}_{\rm{ab}}(S_{ \e\rm{fl}},G\e)\ar[r]& H^{\le
1}_{\rm{ab}}(K_{\e\rm{fl}},G\e),}
$$
where the top arrow is surjective by Corollary 3.6, the left-hand
vertical map is surjective by \cite{ga3}, Proposition 5.5(i), and
the right-hand vertical map (which is the restriction of
${\rm{ab}}^{1}_{\lbe G\lbe/\be K}$ to $\Sha^{\le 1}_{S}(K,G\e)$)
maps $\Sha^{\le 1}_{S}(K,G\e)$ onto $\Sha^{1}_{{\rm{ab}},
S}\e(K,G\e)$ by Lemma 3.3. We conclude that the bottom map in the
above diagram induces a surjection $H^{1}_{\rm{ab}}(S_{
\e\rm{fl}},G\e)\twoheadrightarrow\Sha^{1}_{{\rm{ab}}, S}\e(K,G\e)$.

\begin{definition} Let $G$ be a reductive group scheme over $S$.
The {\it abelian class group of $G$} is the group
$$
C_{\rm{ab}}(G\e)=\krn\!\be\left[H^{1}_{\rm{ab}}(S_{
\e\rm{fl}},G\e)\ra\Sha^{1}_{{\rm{ab}}, S}\e(K,G\e)\right].
$$
\end{definition}

Thus, there exists an exact sequence
\begin{equation}\label{abseq}
1\ra C_{\rm{ab}}(G\e)\ra H^{1}_{\rm{ab}}(S_{
\e\rm{fl}},G\e)\ra\Sha^{1}_{{\rm{ab}}, S}\e(K,G\e) \ra 1.
\end{equation}

\smallskip

\begin{examples} \noindent

\begin{enumerate}

\item[(a)] If $G$ is semisimple with fundamental group $\mu$
then, by Example 2.1(a),
$H^{1}_{\rm{ab}}(S_{\e\rm{fl}},G\e)=H^{2}(S_{ \e\rm{fl}},\mu\e)$ and
$\Sha^{1}_{{\rm{ab}}, S}\e(K,G\e)=\Sha^{2}_{S}(K,\mu)$, where
$$
\Sha^{2}_{S}(K,\mu)=\krn\!\!\left[\e
H^{2}(K_{\rm{fl}},\mu)\ra\displaystyle\prod_{v\in S_{\e 0}}H^{\le
2}(K_{v,{\rm{fl}}},\mu)\right].
$$
Thus there exists an exact sequence of abelian groups
$$
1\ra C_{\rm{ab}}(G\e)\ra H^{2}(S_{
\e\rm{fl}},\mu\e)\ra\Sha^{2}_{S}(K,\mu)\ra 1.
$$
Note that $C_{\rm{ab}}(G\e)$ is annihilated by the exponent of
$\mu$. When $n\geq 2$ is an integer and $\mu=\mu_{n,S}$ is the group
scheme of $n$-th roots of unity on $S$, $C_{\rm{ab}}(G\e)$ can be
computed explicitly. Indeed, taking cohomology of the exact sequence
of fppf sheaves $1\ra \mu_{n}\ra\bg_{m}\overset{\! n}\ra \bg_{m}\ra
1$ over $S$ and over $K$, we obtain an exact commutative diagram
$$
\xymatrix{1\ar[r]&\pic(S)/n\ar[r]\ar[d]&H^{\e
2}\lbe(S_{\e\rm{fl}},\mu_{n})\ar[r]\ar[d]& \br(S)_{n}\ar[r]\ar[d]&1\\
1\ar[r]&1\ar[r]&H^{\e
2}\lbe(K_{\e\rm{fl}},\mu_{n})\ar[r]^(.55){\sim}&
\br(K)_{n}\ar[r]&1.}
$$
The right-hand vertical map in the above diagram is injective by
\cite{adt}, proof of Proposition II.2.1, p.164, line -10, whence
there exists a canonical isomorphism
$$
C_{\rm{ab}}(G\e)=\krn[\e H^{\e 2}\lbe(S_{\e\rm{fl}},\mu_{n})\ra
H^{\e 2}\lbe(K_{\e\rm{fl}},\mu_{n})]\simeq\pic\e(S)/n.
$$

\smallskip

\item[(b)] If $G$ has trivial fundamental group, then $H^{1}_{\rm{ab}}(S_{
\e\rm{fl}},G\e)=H^{1}(S_{ \e\rm{\acute{e}t}},G^{\tor}\e)$ and
$\Sha^{1}_{\rm{ab}, S}(K,G\e)=\Sha^{1}_{S}(K,G^{\tor})$ by Example
2.1(b). Thus, by Corollary 3.6, $C_{\rm{ab}}(G\e)=C(G^{\tor})$. The
latter group is the N\'eron-Raynaud $\Sigma$-class group of the
$K$-torus $G^{\tor}_{\be K}$ introduced in \cite{ga2}, where
$\Sigma$ is the set of primes of $K$ which do not correspond to a
point of $S$. To check this, we only need to show that the $S$-torus
$G^{\tor}$ is the identity component of the N\'eron-Raynaud model
$\s N$ of $G^{\tor}_{\be K}$ over $S$. There exists a unique
$S$-morphism $i\colon G^{\tor}\ra \s N$ which extends the identity
morphism on $G^{\tor}_{\be K}$ (see \cite{blr}, \S10.1, p.289). By
\cite{sga3}, $\text{VI}_{\text{B}}$, Lemma 3.10.1, we only need to
check that $i$ is an open immersion. Let $S^{\e\prime}\ra S$ be a
connected finite \'etale cover of $S$ with function field $L$ such
that $G^{\tor}\times_{S}S^{\e\prime}\simeq\bg_{m,S^{\e\prime}}^{\e
r}$. Then $G^{\tor}\times_{S}S^{\e\prime}\ra\s
N\times_{S}S^{\e\prime}$ is the embedding of
$\bg_{m,S^{\e\prime}}^{\e r}$ into $\s N\times_{S}S^{\e\prime}$,
which is the N\'eron-Raynaud model of $G_{\be K}\times_{\spec
K}\spec L\simeq \bg_{m,\e L}^{\e r}$ over $S^{\e\prime}$ (see
\cite{blr}, \S7.2, Theorem 1(iii), p.176). Thus
$i\times_{S}S^{\e\prime}\colon G^{\tor}\times_{S}S^{\e\prime} \ra\s
N\times_{S}S^{\e\prime}$ is an open immersion and hence so is $i\be$
\footnote{I thank B.Edixhoven for sending me this proof.}.

\end{enumerate}
\end{examples}

Consider the exact commutative diagrams
\begin{equation}\label{abdiag2}
\xymatrix{1\ar[r]&C(G\e)\ar[r]&H^{\e
1}\lbe(S_{\e\rm{\acute{e}t}},G\e)
\ar@{->>}[r]\ar[d]& \Sha^{\le 1}_{S}(K,G\e)\ar[d]\\
&&\displaystyle\prod_{v\le\in\le\Sigma}H^{\le
1}(K_{v},G\e)\ar@{=}[r]&\displaystyle\prod_{v\le\in\le\Sigma}H^{\le
1}(K_{v},G\e)&}
\end{equation}
and
\begin{equation}\label{abdiag3}
\xymatrix{1\ar[r]&C_{\rm{ab}}(G\e)\ar[r]&H^{1}_{\rm{ab}}(S_{
\e\rm{fl}},G\e)\ar@{->>}[r]\ar[d]& \Sha^{1}_{\rm{ab}, S}(K,G\e)
\ar[d]\\
&&\displaystyle\prod_{v\le\in\le\Sigma}H^{\le
1}_{\rm{ab}}(K_{v},G\e)\ar@{=}[r]&\displaystyle\prod_{v\le\in\le\Sigma}H^{\le
1}_{\rm{ab}}(K_{v},G\e)&}
\end{equation}
whose top rows are given by Corollary 3.6 and \eqref{abseq},
respectively. The middle vertical maps in the above diagrams are
induced by the compositions $\spec K_{v}\ra\spec K\ra S$ for
$v\in\Sigma$. Then \eqref{abdiag2} and \eqref{abdiag3} induce exact
sequences of pointed sets
\begin{equation}\label{secu}
1\ra C(G\e)\ra D^{1}\lbe(S,G\e)\ra\Sha^{\le 1}\lbe(K,G\e)\ra 1
\end{equation}
and
\begin{equation}\label{secuab}
1\ra C_{\rm{ab}}(G\e)\ra
D^{1}_{\rm{ab}}\lbe(S,G\e)\ra\Sha^{1}_{\rm{ab}}\lbe(K,G\e) \ra 1,
\end{equation}
where
$$
D^{1}\lbe(S,G\e)=\krn\!\!\left[\e H^{\e
1}\lbe(S_{\e\rm{\acute{e}t}},G\e)
\ra\displaystyle\prod_{v\le\in\le\Sigma}H^{\le 1}(K_{v},G\e)\right]
$$
and
\begin{equation}\label{defs}
D^{1}_{\rm{ab}}\lbe(S,G\e)=\krn\!\!\left[\e H^{\e
1}_{\rm{ab}}\lbe(S_{\e\rm{fl}},G\e)
\ra\displaystyle\prod_{v\le\in\le\Sigma}H^{\le
1}_{\rm{ab}}(K_{v,\rm{fl}},G\e)\right].
\end{equation}
Clearly, the restriction of ${\rm{ab}}^{\lbe 1}_{G/S}\colon H^{\e
1}\lbe(S_{\e\rm{\acute{e}t}},G\e)\ra H^{\e
1}_{\rm{ab}}\lbe(S_{\e\rm{fl}},G\e)$ to $D^{\e 1}\lbe(S,G\e)$
defines a map $D^{\e 1}\lbe(S,G\e)\ra D^{1}_{\rm{ab}}(S,G\e)$.
Further, there exists an exact commutative diagram
$$
\xymatrix{1\ar[r]&C(G\e)\ar[r]&D^{\e 1}\lbe(S,G\e)\ar[r]\ar[d]&
\Sha^{\le 1}\lbe(K,G\e)
\ar[r]\ar[d]^(.5){\simeq}& 1\\
1\ar[r]&C_{\rm{ab}}(G\e)\ar[r]&D^{1}_{\rm{ab}}(S,G\e)\ar[r]&
\Sha^{1}_{\rm{ab}}\lbe(K,G\e) \ar[r]& 1,}
$$
where the right-hand vertical map is induced by ${\rm{ab}}^{\lbe
1}_{G/K}$. That the latter map is bijective is \cite{ga3}, Corollary
5.10. We may now define a map
\begin{equation}\label{map}
C(G\e)\ra C_{\rm{ab}}(G\e)
\end{equation}
to be that induced by the composition $C(G\e)\ra D^{\e
1}\lbe(S,G\e)\ra D^{1}_{\rm{ab}}(S,G\e)$. Note that \eqref{map} is
surjective if, and only if, $C_{\rm{ab}}(G\e)\subset {\rm{ab}}^{\lbe
1}_{G/S}(D^{\e 1}\lbe(S,G\e))$. Now let
$C\big(\Gtil\e\big)^{\prime}$ denote the kernel of the composition
\begin{equation}\label{psi}
H^{\e 1}\be\big(S_{\e\rm{\acute{e}t}},\Gtil\e\big)\overset{\!
\widetilde{\lambda}} \twoheadrightarrow H^{\e
1}\be\big(K,\Gtil\e\big)\overset{\pi}\longrightarrow H^{\e
1}\be\big(K,\Gtil\e\big)/H^{\e 0}_{\rm{ab}}(K_{\rm{fl}},G\e),
\end{equation}
where $\widetilde{\lambda}$ is induced by $\spec K\ra S$ and $\pi$
is the canonical projection. If $\delta_{\le 0}\colon H^{\e
0}_{\rm{ab}}(K_{\rm{fl}},G\e)\ra H^{\e 1}\be\big(K,\Gtil\e\big)$ is
the map appearing in Theorem 2.2(i), then $\krn\pi=\img\delta_{\le
0}$. On the other hand, by Corollary 3.7(i),
$\krn\widetilde{\lambda}$ is in bijection with $C\big(\Gtil\e\big)$.
Thus, by the surjectivity of $\widetilde{\lambda}$, the pair of maps
\eqref{psi} induces an exact sequence of pointed sets
$$
1\ra C\big(\Gtil\e\big)\ra
C\big(\Gtil\e\big)^{\prime}\ra\img\delta_{\le 0}\ra 1,
$$
where the first nontrivial map is injective. Note that, by Corollary
3.7(ii), $\img\delta_{\le 0}$ is in bijection with a subset of
$\prod_{\,v\,\rm{real}}H^{\le 1}(K_{v},G\e)$.

\begin{proposition} There exists an exact sequence of pointed sets
$$
1\ra\mu(S)\ra\Gtil(S)\ra G (S) \ra H^{\e 0}_{\rm{ab}}(S_{\rm{fl}},G
\e)\ra C\big(\Gtil\e\big)^{\prime}\ra C(G\e)\ra C_{\rm{ab}}(G\e),
$$
where $C\big(\Gtil\e\big)^{\prime}$ is the kernel of the composition
\eqref{psi}.
\end{proposition}
\begin{proof} This follows from the bijection $C(G)\simeq\krn
[H^{\e 1}\be(S_{\e\rm{\acute{e}t}},G\e)\ra H^{\e 1}(K,G\e)]$ of
Theorem 3.4 and the exact commutative diagram
$$
\xymatrix{\dots\ar[r]&H^{\e 0}_{\rm{ab}}(S_{\rm{fl}},G
\e)\ar[r]\ar[d]&H^{\e
1}\be\big(S_{\e\rm{\acute{e}t}},\Gtil\e\big)\ar[r]
\ar[d]^(.45){\pi\circ\e\tilde{\lambda}}& H^{\e
1}\be(S_{\e\rm{\acute{e}t}},G\e)\ar[r]\ar[d]& H^{\e
1}_{\rm{ab}}(S_{\e\rm{fl}},G\e)\ar[d]\\
&1\ar[r]&H^{\e 1}\be\big(K,\Gtil\e\big)/H^{\e
0}_{\rm{ab}}(K_{\rm{fl}},G\e)\ar[r]& H^{\e 1}(K,G\e)\ar[r]& H^{\e
1}_{\rm{ab}}(K_{\e\rm{fl}},G\e),}
$$
whose top and bottom rows are given by Theorem 2.2(i) and (ii),
respectively. For the commutativity of the above diagram, see Remark
2.3.
\end{proof}

\begin{theorem} Assume that $K$ has no real primes. Then there exists an exact
sequence of pointed sets
$$
1\ra H^{\e 1}\be\big(S_{\e\rm{\acute{e}t}},\Gtil\e\big)/H^{\e
0}_{\rm{ab}}(S_{\rm{fl}},G \e)\ra C(G\e)\ra C_{\rm{ab}}(G\e)\ra 1,
$$
where the first nontrivial map is injective.
\end{theorem}
\begin{proof} By \cite{ga3}, Theorem 5.8(i) and Remark 5.9(a), the
abelianization map ${\rm{ab}}_{G/K}^{1}\colon H^{\e 1}\be(K,G\e)\ra
H^{\e 1}_{\rm{ab}}(K_{\rm{fl}},G\e)$ is bijective. The theorem is
now immediate from the exact commutative diagram
$$
\xymatrix{1\ar[r]& H^{\e
1}\be\big(S_{\e\rm{\acute{e}t}},\Gtil\e\big)/H^{\e
0}_{\rm{ab}}(S_{\e\rm{fl}},G\e)\ar[d]\ar[r]& H^{\e
1}\be(S_{\e\rm{\acute{e}t}},G\e)\ar@{->>}[r]\ar[d]& H^{\e
1}_{\rm{ab}}(S_{\e\rm{fl}},G\e)\ar[d]\\
&1\ar[r]&H^{\e 1}\be(K,G\e) \ar[r]^(.45){\sim}& H^{\e
1}_{\rm{ab}}(K_{\rm{fl}},G\e),}
$$
whose top row is given by Theorem 2.2(ii).
\end{proof}

\begin{remarks} \indent
\begin{enumerate}
\item[(a)] The maps appearing in the exact sequences of Proposition 3.10 and
Theorem 3.11 are induced by the corresponding maps appearing in the
exact sequence of Theorem 2.2(i), all of which are explicitly
described in \cite{ga3}.

\item[(b)] Assume that $K$ has no real primes and that
$\Gtil_{\be K}$ has the strong approximation property with respect
to $\Sigma$ (see \cite{pr}, \S7.1). Then the set $H^{\e
1}\be\big(S_{\e\rm{\acute{e}t}},\Gtil\e\big)\simeq
C\big(\Gtil\e\big)$ (see Corollary 3.7) is trivial and $C(G)$ is
known to have a natural structure of abelian group (see, e.g.,
\cite{th}, Proposition 13, p.32). This group, which was denoted
${\mathcal G}{\rm{cl}}(G)$ in \cite{pr}, \S8, has been studied in
\cite{pr}, \S8.2, \cite{th}, \S\S4.4– 4.5, and \cite{dem2}, \S4. Now
the theorem and a twisting argument (see \cite{ga3}, Corollary 3.15)
show that there exists an isomorphism of abelian groups ${\mathcal
G}{\rm{cl}}(G)\simeq C_{\rm{ab}}(G\e)$. In particular, by Corollary
3.14 below, ${\mathcal G}{\rm{cl}}(G)$ is endowed with natural {\it
corestriction homomorphisms}. For the case of number fields with
real primes, see \cite{dem2}, Theorem 4.6.

\item[(c)] In general, the map $C(G\e)\ra C_{\rm{ab}}(G\e)$ is not
surjective. More precisely, let $c$ be a class in
$C_{\rm{ab}}(G\e)\subset H^{\e 1}_{\rm{ab}}(S_{\e\rm{fl}},G\e)$ and
consider the exact commutative diagram
$$
\xymatrix{H^{\e
1}\be\big(S_{\e\rm{\acute{e}t}},\Gtil\e\big)\ar[r]^{\partial^{(1)}}
\ar@{->>}[d]^(0.45){\tilde{\lambda}}& H^{\e
1}\be(S_{\e\rm{\acute{e}t}},G\e)\ar[r]\ar[d]^(0.45){\lambda}& H^{\e
1}_{\rm{ab}}(S_{\e\rm{fl}},G\e)\ar[d]\ar[r]&1\\
H^{\e 1}\be\big(K,\Gtil\e\big)\ar[r]^{\partial^{(1)}_{K}}&H^{\e
1}\be(K,G\e) \ar[r]& H^{\e 1}_{\rm{ab}}(K_{\rm{fl}},G\e)\ar[r]&1,}
$$
whose left-hand vertical map is surjective by Corollary 3.7(i). Let
$c^{\e\prime}$ be a preimage of $c$ in $H^{\e
1}\be(S_{\e\rm{\acute{e}t}},G\e)$. Since $c$ maps to zero in $H^{\e
1}_{\rm{ab}}(K_{\rm{fl}},G\e)$,
$\lambda(c^{\e\prime})=\partial^{\e(1)}_{K}(x^{\e\prime})$ for some
$x^{\e\prime}\in H^{\e 1}\be\big(K,\Gtil\e\big)$. Let
$p^{\e\prime}\in H^{\e 1}\be\big(S_{\e\rm{\acute{e}t}},\Gtil\e\big)$
be such that $\tilde{\lambda}(p^{\e\prime})=x^{\e\prime}$ and let
$p=\partial^{\e(1)}(p^{\e\prime})$. Choose a $\Gtil$-torsor
$P^{\e\prime}$ representing $p^{\e\prime}$ and let
$P=P^{\e\prime}\be\wedge^{\Gtil} G$, which is a $G$-torsor
representing $p$. Since $\lambda(c^{\e\prime})=\lambda(p)$, Remark
3.5 shows that
$c^{\e\prime}=(\theta_{P}^{-1}\circ\be\!\phantom{.}^{P}\!\varphi)(c^{\e\prime\prime})$
for some class $c^{\e\prime\prime}\in C(\be\!\phantom{.}^{P}\be G)$.
Now, by \cite{ga3}, Lemma 3.13, we have
$$
\big(\!\phantom{.}^{P}\!{\rm{ab}}^{\lbe 1}_{G/\lbe
S}\circ\be\!\phantom{.}^{P}\!\varphi\e\big)(c^{\e\prime\prime})=
\big(\!\phantom{.}^{P}\!{\rm{ab}}^{\lbe 1}_{G/\lbe S}\circ
\theta_{P}\e\big)(c^{\e\prime})={\rm{ab}}^{\lbe 1}_{G/\lbe
S}(c^{\e\prime})=c.
$$
Thus, the following holds. Let $\s S$ be a complete set of
representatives for the classes in $H^{\e
1}\be\big(K,\Gtil\e\big)\simeq\prod_{\,v\,\text{real}}H^{\e
1}\be\big(K_{v},\Gtil\e\big)$. For each $\Gtil_{\lbe\lbe K}$-torsor
$x\in\s S$, choose an extension of $x$ to a $\Gtil$-torsor $X$ and
let $\!\phantom{.}^{x}\be G$ denote the $(X\!\wedge^{\Gtil}\!
G\e)$-twist of $G$. Then there exists a surjection $\coprod_{\e
x\in\s S}\be C(\be\!\phantom{.}^{x}\be G)\twoheadrightarrow
C_{\rm{ab}}(G)$. Consequently, since each set
$C(\be\!\phantom{.}^{x}\be G)$ is in bijection with $C(G)$, we
conclude that there exists a surjection
$$
\coprod_{\e x\in\s S}\be C(G)\twoheadrightarrow C_{\rm{ab}}(G).
$$
In particular, $\mid\! C_{\rm{ab}}(G)\!\mid$ is a lower bound for
$\prod_{\,v\,\text{real}}\#H^{\e 1}\be\big(K_{v},\Gtil\e\big)\cdot
\#C(G)$.

\item[(d)] As mentioned in Remark 2.3, the right action of $H^{\e
0}_{\rm{ab}}(S_{\rm{fl}},G \e)$ on $H^{\e
1}\be\big(S_{\e\rm{\acute{e}t}},\Gtil\e\big)$ is compatible with
inverse images. Thus, via the bijection
$C\big(\Gtil\e\big)\simeq\krn[H^{\e
1}\be\big(S_{\e\rm{\acute{e}t}},\Gtil\e\big)\ra H^{\e
1}\be\big(K,\Gtil\e\big)]$ of Theorem 3.4, it induces a right action
of the abelian group\footnote{The isomorphism follows from
\eqref{kamb}.}
$$\begin{array}{rcl}
C_{\rm{ab}}^{\e 0}(G\e)&=&\krn\!\be\left[H^{0}_{\rm{ab}}(S_{
\e\rm{fl}},G\e)\ra H^{0}_{\rm{ab}}\e(K_{
\e\rm{fl}},G\e)\right]\\
&\simeq&\krn\!\be\left[H^{1}(S_{ \e\rm{fl}},\mu\e)\ra H^{1}(K_{
\e\rm{fl}},\mu\e)\right]
\end{array}
$$
on the set $C\big(\Gtil\e\big)$ . Using \cite{ga3}, Proposition
3.14(a), it can be shown that the stabilizer in $C_{\rm{ab}}^{\e
0}(G\e)$ of a class $p\in C\big(\Gtil\e\big)$ represented by a
$\Gtil$-torsor $P$ is
$$
\!\phantom{.}^{Q}\lbe{\rm{ab}}^{\lbe 0}_{G/\lbe
S}\big(\!\phantom{.}^{Q}\lbe G(S)\cap
\!\phantom{.}^{Q}\be\partial\big(\!\phantom{.}^{Q}\lbe
\Gtil\big(K\big)\big)\big)\subset C_{\rm{ab}}^{\e 0}(G\e),
$$
where $Q=P\be\wedge^{\Gtil}\be G$, $\!\phantom{.}^{Q}\lbe\partial$
is the $Q$-twist of $\partial\colon \Gtil\ra G$ and the intersection
takes place inside $\!\phantom{.}^{Q}\lbe G(K)$. In the interesting
particular case where $\mu=\mu_{n,S}$ is the group scheme of $n$-th
roots of unity on $S$ ($n\geq 2$), the cohomological argument given
in Example 3.9(a) yields an injection $C_{\rm{ab}}^{\e
0}(G\e)\hookrightarrow\text{Pic}(S)_{n}$.

\end{enumerate}

\end{remarks}

\smallskip

In the remainder of this Section we establish some basic properties
of $C_{\rm{ab}}(G\e)$ using flasque resolutions of $G$. Recall that
a flasque resolution of $G$ is an exact sequence $1\ra F\ra H \ra G
\ra 1$, where $F$ is a flasque $S$-torus, $R=H^{\tor}$ is a
quasi-trivial $S$-torus and $H^{\der}$ is a semisimple and
simply-connected $S$-group scheme. See \cite{ga4} for more details.

\smallskip

For any torus $T$ over $S$, set
$$
D^{\le 2}\lbe(S,T\e)=\krn\!\!\left[\e H^{\e
2}\lbe(S_{\e\rm{\acute{e}t}},T\e)
\ra\displaystyle\prod_{v\e\in\e\Sigma}H^{\le 2}(K_{v},T\e)\right],
$$
where the map involved is induced by the compositions $\spec
K_{v}\ra\spec K\ra S$ for $v\e\in\e\Sigma$. If $v\in S_{0}$, the
preceding composition coincides with the composition $\spec
K_{v}\ra\spec\mathcal O_{v}\ra S$. Since $H^{\e 2}\lbe(\mathcal
O_{v,\e\rm{\acute{e}t}},T\e)=H^{\e 2}\lbe(k(v),T\e)=0$ by
\cite{mi1}, III.3.11(a), p.116 (recall that $k(v)$-tori are
cohomologically trivial), we conclude that the map $H^{\e
2}\lbe(S_{\e\rm{\acute{e}t}},T\e)\ra H^{\e 2}\lbe(K,T\e)$ induces a
map $D^{2}\lbe(S,T\e)\ra\Sha^{2}(K,T)$.

\begin{theorem} Let $G$ be a reductive group scheme over $S$ and let
$1\ra F\ra H \ra G \ra 1$ be a flasque resolution of $G$. Set
$R=H^{\lbe\tor}$. Then the given resolution induces an exact
sequence of abelian groups
$$
\mu(S)\hookrightarrow F(S)\ra R(S)\ra
H^{0}_{\rm{ab}}(S_{\rm{fl}},G\le)\ra H^{\le
1}(S_{\e\rm{\acute{e}t}},F\e)\ra C(R)\ra C_{\rm{ab}}(G\e)\ra 1.
$$
\end{theorem}
\begin{proof} Since $R$ is quasi-trivial, Theorem 3.4 and Hilbert's
Theorem 90 show that $C(R)=H^{\le 1}\lbe(S_{\e\rm{\acute{e}t}},R)$.
Thus, since $H^{-1}_{\rm{ab}}(S_{\rm{fl}},G\le)=\mu(S)$ by
\cite{ga3}, (2.1), the sequence of the theorem is exact up to the
term $H^{\le 1}(S_{\e\rm{\acute{e}t}},F\e)$ by \cite{ga4},
Proposition 4.2. Thus, it remains only to check the exactness of the
sequence $H^{\le 1}(S_{\e\rm{\acute{e}t}},F\e)\ra C(R)\ra
C_{\rm{ab}}(G\e)\ra 1$. By \cite{ga4}, Propositions 4.2 and 4.5,
there exists an exact commutative diagram
$$
\xymatrix{\dots\ar[r]&C(R)\ar[r]& H^{\e
1}_{\rm{ab}}(S_{\e\rm{fl}},G\e)\ar[r]\ar[d]& H^{\e
2}\lbe(S_{\e\rm{\acute{e}t}},F\e)\ar[r]\ar[d]& H^{\e
2}\lbe(S_{\e\rm{\acute{e}t}},R\e)\ar[d]\\
&&\displaystyle\prod_{v\le\in\le\Sigma}H^{\e
1}_{\rm{ab}}\be(K_{v},G\e)\ar@{^{(}->}[r]&
\displaystyle\prod_{v\le\in\le\Sigma}H^{\e 2}(K_{v},F\e)\ar[r]&
\displaystyle\prod_{v\le\in\le\Sigma}H^{\e 2}(K_{v},R\e).}
$$
The right-hand vertical map in the above diagram is injective by
\cite{adt}, Proposition II.2.1, p.163. Thus the preceding diagram
yields an exact sequence
$$
H^{\le 1}(S_{\e\rm{\acute{e}t}},F\e)\ra C(R)\ra
D^{1}_{\rm{ab}}(S,G\e)\ra D^{2}\lbe(S,F\e)\ra 1.
$$
Now, since $F$ is flasque, the map $H^{\e
2}\lbe(S_{\e\rm{\acute{e}t}},F\e)\ra H^{\e 2}\lbe(K,F\e)$ is
injective (see \cite{cts2}, Theorem 2.2(ii), p.161), whence
$D^{2}\lbe(S,F\e)\ra\Sha^{2}(K,F)$ is injective as well. The theorem
now follows from the diagram
$$
\xymatrix{H^{\le 1}(S_{\e\rm{\acute{e}t}},F\e)\ar[r]& C(R)\ar[r]&
D^{1}_{\rm{ab}}(S,G\e)\ar[r]\ar[d]&D^{2}\lbe(S,F\e)\ar@{^{(}->}[d]
\ar[r]&1\\
&&\Sha^{1}_{\rm{ab}}\lbe(K,G\e)\ar[r]^{\simeq}&\Sha^{2}(K,F),&}
$$
where the bottom map is an isomorphism by Proposition 3.1.
\end{proof}

\begin{corollary} Let $L/K$
be a finite Galois extension and let $S^{\e\prime}\ra S$ be the
normalization of $S$ in $L$. Let $G$ be a reductive group scheme
over $S$ and let $1\ra F\ra H \ra G \ra 1$ be a flasque resolution
of $G$. Then the given resolution defines a {\rm corestriction
homomorphism}
$$
{\rm{cores}}_{ S^{\prime}\be/S}\colon
C_{\rm{ab}}(S^{\e\prime},G\e)\ra C_{\rm{ab}}(S,G\e),
$$
where $C_{\rm{ab}}(S^{\e\prime},G\e)$ (respectively,
$C_{\rm{ab}}(S,G\e)$) is the abelian class group of
$G\times_{S}S^{\e\prime}$ (respectively, $G$).
\end{corollary}
\begin{proof} Let $R=H^{\tor}$ and recall that $C(S,R)=H^{\le
1}(S_{\e\rm{\acute{e}t}},R\e)$ (and similarly for
$R\times_{S}S^{\e\prime}\e$). There exists a canonical corestriction
homomorphism
$$
{\rm{cores}}_{S^{\prime}/S}\colon H^{\le 1}(S^{\e\prime}
_{\e\rm{\acute{e}t}},F\e)\ra H^{\le 1}(S_{\e\rm{\acute{e}t}},F\e),
$$
namely the composite
$$
H^{\le i}(S^{\e\prime}_{\e\rm{\acute{e}t}},F\e)\overset{\sim} \ra
H^{\le i}\big(S_{\e\rm{\acute{e}t}},R_{S^{\prime}/S}\big(F\times_{S}
S^{\e\prime}\e\big)\e\big)\ra H^{\le i}(S_{\e\rm{\acute{e}t}},F\e),
$$
where the second map is induced by the trace morphism $R
_{S^{\prime}/S}\big(F\times_{S} S^{\e\prime}\e\big)\ra F$ of
\cite{sga4}, XVII, 6.3.13.2. See \cite{cts2}, (0.4.1), p.154.
Similarly, there exists a corestriction homomorphism
${\rm{cores}}_{S^{\prime}/S}\colon C(S^{\e\prime},R)\ra C(S,R)$.
Now, by \cite{cts2}, Proposition 1.4, p.158, $1\ra
F\times_{S}S^{\e\prime}\ra H\times_{S}S^{\e\prime} \ra
G\times_{S}S^{\e\prime} \ra 1$ is a flasque resolution of
$G\times_{S}\e S^{\e\prime}$, and the theorem yields an exact
commutative diagram
$$
\xymatrix{H^{\le
1}(S^{\e\prime}_{\e\rm{\acute{e}t}},F\e)\ar[r]\ar[d]^{{\rm{cores}}_{
S^{\prime}\be/S}}&C(S^{\e\prime},R)\ar[d]^{{\rm{cores}}_{
S^{\e\prime}\be/S}}\ar[r]&
C_{\rm{ab}}(S^{\e\prime},G\e)\ar[r]\ar@{-->}[d]& 1\\
H^{\le 1}(S_{\e\rm{\acute{e}t}},F\e)\ar[r]&C(S,R)\ar[r]&
C_{\rm{ab}}(S,G\e)\ar[r]& 1.}
$$
This establishes the existence of the right-hand vertical map in the
above diagram, which is the assertion of the corollary.
\end{proof}

\begin{remark} The map of the corollary is independent, up to
isomorphism, of the chosen flasque resolution of $G$, i.e., if $1\ra
F_{1}\ra H_{1} \ra G \ra 1$ is another flasque resolution of $G$ and
${\rm{cores}}_{ S^{\prime}\be/S,1}\colon
C_{\rm{ab}}(S^{\e\prime},G\e)\ra C_{\rm{ab}}(S,G\e)$ is the
corestriction map that it defines, then \cite{ga4}, Proposition
3.7(iii), shows that there exist automorphisms $\sigma^{\e\prime}$
of $C_{\rm{ab}}(S^{\e\prime},G\e)$, $\sigma$ of $C_{\rm{ab}}(S,G\e)$
and a commutative diagram
$$
\xymatrix{C_{\rm{ab}}(S^{\e\prime},G\e)\ar[d]^{{\rm{cores}}
_{S^{\prime}/S}}\ar[r]^{\sigma^{\e\prime}}&C_{\rm{ab}}(S^{\e\prime},G\e)
\ar[d]^{{\rm{cores}}_{S^{\prime}/S,1}}\\
C_{\rm{ab}}(S,G\e)\ar[r]^{\sigma}&C_{\rm{ab}}(S,G\e).}
$$
Further, ${\rm{cores}} _{S^{\prime}/S}$ is functorial in
$S^{\prime}\ra S$. This follows from the functoriality of the
corestriction homomorphisms in the case of tori mentioned above.
\end{remark}

Recall that a flasque $S$-torus $F$ is called {\it invertible} if it
is a direct factor of a quasi-trivial $S$-torus.

\begin{corollary} Assume that $G$ admits an {\rm{invertible}}
resolution, i.e., there exists a flasque resolution $1\ra F\ra H \ra
G \ra 1$ with $F$ invertible\,\footnote{This is the case if both
$G^{\tor}$ and $Z\big(\Gtil\e\big)$ are split by a metacyclic Galois
cover of $S$. See \cite{ga4}, Remark 3.3.}. Then the given
resolution induces an exact sequence of abelian groups
$$
1\ra\mu(S)\ra F(S)\ra R(S)\ra H^{0}_{\rm{ab}}(S_{\rm{fl}},G\le)\ra
C(F)\ra C(R) \ra C_{\rm{ab}}(G\e)\ra 1,
$$
where $R=H^{\tor}$.
\end{corollary}
\begin{proof} Since $F$ is
invertible, $H^{1}(K,F)=0$ by Hilbert's Theorem 90. Thus $H^{\le
1}(S_{\e\rm{\acute{e}t}},F\e)=C(F)$ by Theorem 3.4 and the corollary
is now immediate from the theorem.
\end{proof}

\begin{remark} The corollary extends the ``good reduction" case of
\cite{ga2}, Theorem 6.1, from $K$-tori to arbitrary connected
reductive $K$-groups. Indeed, let $T$ be a $K$-torus with
multiplicative reduction over $S$ which admits an invertible
resolution $1\ra F\ra R \ra T \ra 1$, i.e., $F$ is invertible and
$R$ is quasi-trivial (this is the case, for example, if $T$ is split
by a metacyclic extension of $K$, by \cite{cts1}, Proposition 2,
p.184). Let $j\colon\spec K\ra S$ be the canonical morphism. Since
$R^{1}j_{*}F=0$ for the smooth topology on $S$ (see \cite{ga2},
Lemma 4.8(b)), the given resolution induces an exact sequence
$1\ra\s F\overset{\iota}\ra\s R \ra\s T \ra 1$ of N\'eron-Raynaud
models over $S$. The latter sequence induces, in turn, an exact
sequence $1\ra\iota^{-1}(\s F^{\circ})\ra \s R^{\e\circ}\ra\s
T^{\circ}\ra 1$, where $\s T^{\circ}$ (respectively, $\s
R^{\e\circ},\s F^{\circ}$) denotes the identity component of $\s T$
(respectively, $\s R,\s F$). Further, $\iota^{-1}(\s F^{\circ})/\s
F^{\circ}\simeq\krn[\e\Phi(\s F\e)\ra\Phi(\s R\e)]$, where $\Phi(\s
F\e)=\s F/\s F^{\circ}$ and $\Phi(\s R\e)=\s R/\s R^{\e\circ}$. See
\cite{b}, Theorems 2.2.4 and 2.3.1, pp.49-50. On the other hand, by
\cite{b}, Theorem 2.3.4, p.52, $\krn[\e\Phi(\s F\e)\ra\Phi(\s
R\e)]\subset\Phi(\s F)_{\text{tors}}$, which is zero by \cite{ga2},
Lemma 4.8(c). Thus the given invertible resolution of $T$ induces an
exact sequence of connected N\'eron-Raynaud models
$$
1\ra\s F^{\circ}\ra \s R^{\e\circ}\ra\s T^{\circ}\ra 1.
$$
Since $T$ has multiplicative reduction over $S$, $\s T^{\circ}$ is
an $S$-torus and the preceding sequence is an invertible resolution
of $\s T^{\circ}$. Thus the corollary yields an exact sequence
$$
1\ra\s F^{\circ}(S)\ra\s R^{\e\circ}(S)\ra \s T^{\circ}(S)\ra C(\s
F^{\circ}) \ra C(\s R^{\e\circ})\ra C(\s T^{\circ})\ra 1.
$$
The groups $C(\s T^{\circ}),C(\s R^{\e\circ})$ and $C(\s F^{\circ})$
are the N\'eron-Raynaud class group of $T,R$ and $F$ over $S$,
respectively (see Example 3.9(b)), and the last exact sequence
coincides with the exact sequence of \cite{ga2}, Theorem 6.1, for
the $K$-torus $T$.
\end{remark}

\section{The duality theorem}

Let $K$ and $S$ be as in the Introduction. If $T$ is an $S$-torus
such that $\Sha^{1}_{S}(K,T)=0$, then $C(T)=H^{\e
1}(S_{\rm{\acute{e}t}},T)$ (see Theorem 3.4) is known to satisfy a
duality theorem. Namely, there exists a perfect pairing of finite
groups
\begin{equation}\label{pair}
C(T)\times H^{\e 2}_{\rm{c}}(S_{\rm{\acute{e}t}},T^{*})\ra\bq/\bz
\end{equation}
induced by the natural pairing $T\times T^{*}\ra\bg_{m,S}$, where
$H_{\rm{c}}$ denotes cohomology with compact support. See
\cite{adt}, Theorem  II.4.6(a), p.191. The purpose of this Section
is to extend the above duality theorem to an arbitrary reductive
group scheme $G$ over $S$, with $C_{\rm{ab}}(G\e)$ replacing $C(T)$.
See Theorem 4.12. In particular, Corollary 4.14 below extends
\eqref{pair} to an arbitrary $S$-torus $T$ (see Remark 4.15) .

\smallskip

For any prime number $\ell$, abelian group $B$ and positive integer
$m$, we will write $B_{\ell^{\le m}}$ for the $\ell^{\e m}$-torsion
subgroup of $B$ and $B/\ell^{\e m}$ for $B/\ell^{\e m}B$. Set
$B(\ell)=\cup_{\e m\geq 1}B_{\ell^{\e m}}$,
$B^{(\ell)}=\varprojlim_{m}B/\ell^{\e m}$ and $T_{\ell}\e
B=\varprojlim_{m}B_{\e\ell^{\le m}}$. Further, set $B_{\e
\ell\e\text{-div}}=\bigcap_{\, m} \ell^{\e m}B$ and
$B/\ldiv=B/B_{\ell\e\text{-div}}$.

If $B$ is an abelian topological group and $\bq/\e\bz$ is given the
discrete topology, let
$B^{D}=\text{Hom}_{\e\text{conts.}}(B,\bq/\e\bz)$ be the Pontryagin
dual of $B$. It is endowed with the compact-open topology. If $B$ is
discrete and torsion (respectively, profinite), then $B^{D}$ is
profinite (respectively, discrete and torsion). A continuous pairing
of topological abelian groups $A\times B\ra\bq/\e\bz$ is called {\it
non-degenerate on the right} (respectively,  {\it left}\,) if the
induced homomorphism $B\ra A^{D}$ (respectively, $A\ra B^{D}$) is
injective. It is called {\it non-degenerate} if it is non-degenerate
both on the right and on the left. The pairing is said to be {\it
perfect} if the homomorphisms $B\ra A^{D}$ and $A\ra B^{D}$ are
(topological) isomorphisms. It is not difficult to see that, for any
prime $p$, a perfect pairing $A\times B\ra\bq/\e\bz$ induces
pairings $A(p)\times(B/p\e\text{-div})\ra\bq/\e\bz$ and
$(A/p\e\text{-div})\times B(p)\ra \bq/\e\bz$ which are
non-degenerate on the left and on the right, respectively. Further,
two pairings $(-,-), (-,-)^{\e\prime}\colon A\times B\ra\bq/\e\bz$
are said to be {\it isomorphic} if there exist automorphisms
$\alpha$ of $A$ and $\beta$ of $B$ such that
$(a,b)^{\e\prime}=(\alpha(a),\beta(b))$ for all $a\in A$ and $b\in
B$.

By \cite{hs}, beginning of \S3, and \cite{ga1}, beginning of \S5,
for any cohomologically bounded complex $C$ of abelian sheaves on
$S_{\rm{fl}}$ there exist hypercohomology groups with compact
support ${\bh}^{\e i}_{\e\text{c}}(S_{\rm{fl}},C\e)$ which fit into
an exact sequence\footnote{Recall that, if $v$ is a real prime,
${\bh}^{\e i}(K_{v,\rm{fl}},C\e)$ denotes the $i$-th modified (Tate)
hypercohomology group of $C\times_{S}\spec K_{v}$.}
$$
\dots\ra\! {\bh}^{\e i}_{\e\text{c}}(S_{\rm{fl}},C\e)\!\ra\!
{\bh}^{\e
i}(S_{\rm{fl}},C\e)\!\ra\!\displaystyle\prod_{v\le\in\le\Sigma}{\bh}^{\e
i}(K_{v,\rm{fl}},C\e)\!\ra\! {\bh}^{\e
i+1}_{\e\text{c}}(S_{\rm{fl}},C\e)\!\ra\dots,
$$
where $\Sigma$ is the set of primes of $K$ that do not correspond to
a point of $S$. Set
\begin{equation}\label{dsc}
\begin{array}{rcl}
\bd^{\e i}\lbe(S,C\e)&=&\krn\!\!\left[\e\bh^{\e
i}\lbe(S_{\e\rm{fl}},C\e)\ra\displaystyle\prod_{v\le\in\le\Sigma}\bh^{\le i}(K_{v,\rm{fl}},C\e)\right]\\\\
&=&\img\!\left[\e \bh^{\e i}_{\e\text{c}}(S_{\e\rm{fl}},C\e) \ra
\bh^{\e i}(S_{\e\rm{fl}},C\e)\right]
\end{array}
\end{equation}
and
\begin{equation}\label{shac}
\Sha^{\e i}\lbe(K,C\e)=\krn\!\!\left[\e\bh^{\e
i}\lbe(K_{\rm{fl}},C\e)\ra\displaystyle\prod_{\text{all
$v$}}\bh^{\le i} (K_{v,\rm{fl}},C\e)\right].
\end{equation}

Now let $T_{1}$ and $T_{2}$ be arbitrary $S$-tori and let
$C=(T_{1}\ra T_{2})$, where $T_{1}$ and $T_{2}$ are placed in
degrees $-1$ and $0$, respectively. Let $C^{*}=(T_{2}^{*}\ra
T_{1}^{*})$ be the dual complex of twisted-constant $S$-groups,
where $T_{2}^{*}$ and $T_{1}^{*}$ are placed in degrees $-1$ and
$0$, respectively. Then the canonical morphism
$C\otimes^{\e\mathbf{L}}C^{*}\ra\bg_{m}[1]$ defined in \cite{dem1},
beginning of \S2, induces a pairing\footnote{Recall that, since the
Cartier dual of an $S$-group scheme of multiplicative type is
\'etale, the groups ${\bh}^{\e
2-i}_{\e\text{c}}(S_{\rm{\acute{e}t}},C^{*}\e)$ and ${\bh}^{\e
2-i}_{\e\text{c}}(S_{\rm{fl}},C^{*}\e)$ are canonically isomorphic.}
\begin{equation}\label{p1}
\langle-,-\rangle\colon{\bh}^{\e i}(S_{\rm{fl}},C\e)\times{\bh}^{\e
2-i}_{\e\text{c}}(S_{\rm{\acute{e}t}},C^{*}\e)\ra \bq/\bz
\end{equation}
(cf. \cite{hs}, p.108). The next result extends \cite{dem1},
Corollary 4.7, to the function field case.

\begin{proposition} The pairing \eqref{p1} induces a perfect pairing
of finite groups
$$
\bd^{\e 1}\lbe(S,C\e)\times \bd^{\e 1}\lbe(S,C^{\le *})\ra\bq/\e\bz.
$$
\end{proposition}
\begin{proof} The pairing is defined as follows: if $a\in \bd^{\e
1}\lbe(S,C\e)\subset \bh^{\e 1}\lbe(S_{\e\rm{fl}},C\e)$ and
$a^{\e\prime}\in \bd^{\e 1}\lbe(S,C^{*}\e)\subset \bh^{\e
1}\lbe(S_{\rm{\acute{e}t}},C^{*}\e)$ is the image of
$b^{\e\prime}\in\bh^{\e
1}_{\e\text{c}}(S_{\rm{\acute{e}t}},C^{*}\e)$, then
$\{a,a^{\e\prime}\}= \langle\e a,b^{\e\prime}\e\rangle$, where
$\langle-,-\rangle$ is the pairing \eqref{p1} for $i=1$. We begin by
proving the finiteness statement. As noted in \cite{dem1}, proof of
Corollary 4.7, the finiteness of $\bd^{\e 1}\lbe(S,C\e)$ follows
from (a) the finiteness of $\Sigma$, (b) \cite{adt}, Corollary
I.2.4, p.29, and Theorem II.4.6(a), p.191, and (c) the finiteness of
$D^{2}(S,T)$, where $T$ is an $S$-torus. To prove the latter, assume
first that $T$ is flasque. Then $D^{2}(S,T)$ injects into
$\Sha^{2}(K,T)$ (see the proof of Theorem 3.13), which is finite by
\cite{oes}, Theorem 2.7(a), p.52. In the general case, let $1\ra
T\ra F\ra P\ra 1$ be a flasque resolution of $T$, where $F$ is
flasque and $P$ is quasi-trivial (see \cite{cts2}, (1.3.2), p.158).
Then there exists an exact commutative diagram
$$
\xymatrix{H^{1}(S_{\rm{\acute{e}t}},P)\ar[r]& H^{\e
2}(S_{\rm{\acute{e}t}},T\e)\ar[r]\ar[d]& H^{\e
2}\lbe(S_{\rm{\acute{e}t}},F\e)\ar[d]\\
1\ar[r]&\displaystyle\prod_{v\le\in\le\Sigma}H^{\e
2}(K_{v},T\e)\ar[r]& \displaystyle\prod_{v\le\in\le\Sigma}H^{\e
2}(K_{v},F\e)}
$$
which yields an exact sequence $H^{1}(S_{\rm{\acute{e}t}},P)\ra
D^{2}(S,T)\ra D^{2}(S,F)$. Since $H^{1}(S_{\rm{\acute{e}t}},P)$ and
$D^{2}(S,F)$ are both finite, $D^{2}(S,T)$ is finite as well. As
regards the finiteness of $\bd^{\e 1}\lbe(S,C^{*}\e)$, it again
follows from \cite{adt}, Corollary I.2.4, p.29, the finiteness of
$\Sigma$ and the finiteness of both $D^{2}(S,T_{2}^{*})$ (see
\cite{hs}, proof of Proposition 3.7, p.111), and
$H^{1}(S,T_{1}^{*})$ (see [op.cit.], proof of Lemma 3.2(3), p.108).
Now, it was shown in \cite{dem1}, Corollary 4.7, that the pairing
$$
\bd^{\e 1}\lbe(S,C\e)(\ell)\times \bd^{\e
1}\lbe(S,C^{*}\e)(\ell)\ra\bq/\e\bz,
$$
induced by the pairing of the statement, is perfect for any prime
$\ell\neq p$, where $p=\text{char}\e K$ in the function field case.
The following lemmas will show that it is perfect for $\ell=p$ as
well.

\begin{lemma} For every $i\in\bz$ and every $m\geq 1$, there
exists a perfect pairing
$$
\bh^{\e i}\big(S_{\rm{fl}},C\otimes^{\e\mathbf{L}}\bz/p^{\le
m}\big)\times \bh^{\e
1-i}_{\e\text{c}}\big(S_{\rm{\acute{e}t}},C^{*}\otimes^{\e\mathbf{L}}\bz/p^{\le
m}\big)\ra\bq/\bz,
$$
where the left-hand group is discrete and torsion and the right-hand
group is profinite.
\end{lemma}
\begin{proof} The proof is similar to the proof of \cite{dem1}, Proposition 4.2,
using \cite{adt}, Theorem III.8.2, p.290, in place of [op.cit.],
Corollary II.3.3(b), p.177.
\end{proof}

Now define\footnote{Note that the maps $\bz/p^{m}\ra \bz/p^{m+1}$
(respectively, $\bz/p^{m+1}\ra \bz/p^{m}$) are induced by
multiplication by $p$ on $\bz$ (respectively, the identity map of
$\bz$).}
$$
\bh^{\e i}\big(S_{\rm{fl}},C\otimes^{\e\mathbf{L}} \bq_{\e
p}/\e\bz_{\be p}\big)=\varinjlim_{m}\e \bh^{\e
i}\big(S_{\rm{fl}},C\otimes^{\e\mathbf{L}}\bz/p^{\le m}\big)
$$
and
$$
\bh^{\e i}\big(S_{\rm{fl}},C\otimes^{\e\mathbf{L}}\bz_{\be
p}\big)=\varprojlim_{m}\e \bh^{\e
i}\big(S_{\rm{fl}},C\otimes^{\e\mathbf{L}}\bz/p^{\le m}\big).
$$
Similar definitions apply with $\bh^{\e i}_{\e\text{c}}$ in place of
$\bh^{\e i}$. Then the previous lemma yields perfect pairings
$$
\bh^{\e i}\big(S_{\rm{fl}},C\otimes^{\e\mathbf{L}} \bq_{\e
p}/\e\bz_{\be p}\big)\times \bh^{\e
1-i}_{\e\text{c}}\big(S_{\rm{\acute{e}t}},C^{*}\otimes^{\e\mathbf{L}}\bz_{\be
p}\big)\ra\bq/\bz
$$
and
$$
\bh^{\e i}\big(S_{\rm{fl}},C\otimes^{\e\mathbf{L}} \bz_{\be
p}\big)\times \bh^{\e
1-i}_{\e\text{c}}\big(S_{\rm{\acute{e}t}},C^{*}\otimes^{\e\mathbf{L}}
\bq_{\e p}/\e\bz_{\be p}\big)\ra\bq/\bz.
$$
Consequently, there exist pairings
\begin{equation}\label{p2}
\left(\bh^{\e i}\big(S_{\rm{fl}},C\otimes^{\e\mathbf{L}} \bq_{\e
p}/\e\bz_{\be p}\big)/\pdiv\right)\times\bh^{\e
1-i}_{\e\text{c}}\big(S_{\rm{\acute{e}t}},C^{*}\otimes^{\e\mathbf{L}}\bz_{\be
p}\big)(p)\ra\bq/\bz
\end{equation}
and
\begin{equation}\label{p3}
\bh^{\e i}\big(S_{\rm{fl}},C\otimes^{\e\mathbf{L}} \bz_{\be
p}\big)(p)\times \left(\bh^{\e
1-i}_{\e\text{c}}\big(S_{\rm{\acute{e}t}},C^{*}\otimes^{\e\mathbf{L}}
\bq_{\e p}/\e\bz_{\be p}\big)/\pdiv\right)\ra\bq/\bz
\end{equation}
which are nondegenerate on the right and on the left, respectively.

\begin{lemma} For any $i$, the pairing \eqref{p1} induces a nondegenerate
pairing
$$
\left({\bh}^{\e
i}(S_{\rm{fl}},C\e)(p)/\pdiv\right)\times\left({\bh}^{\,
2-i}_{\e\text{c}}(S_{\rm{\acute{e}t}},C^{*}\e)(p)/\pdiv\right)\ra\bq/\e\bz.
$$
\end{lemma}
\begin{proof} The proof that follows is very similar to the proof of
\cite{ga1}, Theorem 5.2, to which the reader is referred for further
details. There exist an injection ${\bh}^{\e
i}(S_{\rm{fl}},C\e)(p)/\pdiv\hookrightarrow{\bh}^{\e
i}(S_{\rm{fl}},C)^{(p)}(p)$ (see [op.cit.], Lemma 2.1) and an exact
sequence
$$
1\ra{\bh}^{\e i}(S_{\rm{fl}},C)^{(p)}\ra \bh^{\e
i}\big(S_{\rm{fl}},C\otimes^{\e\mathbf{L}}\bz_{\be p}\big)\ra
T_{p}\,{\bh}^{\e i+1}(S_{\rm{fl}},C).
$$
See \cite{dem1}, p.16. Therefore ${\bh}^{\e
i}(S_{\rm{fl}},C)^{(p)}(p)=\bh^{\e
i}\big(S_{\rm{fl}},C\otimes^{\e\mathbf{L}}\bz_{\be p}\big)(p)$ and
the left-hand nondegeneracy of \eqref{p3} yields an injection
$$
{\bh}^{\e i}(S_{\rm{fl}},C\e)(p)/\pdiv\hookrightarrow\left(\bh^{\e
1-i}_{\e\text{c}}\big(S_{\rm{fl}},C^{*}\otimes^{\e\mathbf{L}}
\bq_{\e p}/\e\bz_{\be p}\big)/\pdiv\right)^{D}.
$$
On the other hand, there exists an exact sequence
$$
{\bh}_{\e\text{c}}^{\e
1-i}(S_{\rm{\acute{e}t}},C^{*})\,\otimes\bq_{\e p}/ \,\bz_{\be
p}\hookrightarrow \bh^{\e
1-i}_{\e\text{c}}\big(S_{\rm{\acute{e}t}},C^{*}\otimes^{\e\mathbf{L}}
\bq_{\e p}/\e\bz_{\be p}\big)\twoheadrightarrow{\bh}^{\e
2-i}_{\e\text{c}}(S_{\rm{\acute{e}t}},C^{*})(p)
$$
which identifies ${\bh}^{\e
2-i}_{\e\text{c}}(S_{\rm{\acute{e}t}},C^{*})(p)/\pdiv$ and $\bh^{\e
1-i}_{\e\text{c}}\big(S_{\rm{\acute{e}t}},C^{*}\otimes^{\e\mathbf{L}}
\bq_{\e p}/\e\bz_{\be p}\big)/\pdiv$. Thus we obtain an injection
$$
{\bh}^{\e i}(S_{\rm{fl}},C\e)(p)/\pdiv\hookrightarrow\left({\bh}^{\e
2-i}_{\e\text{c}}(S_{\rm{\acute{e}t}},C^{*})(p)/\pdiv\right)^{D},
$$
i.e., the pairing of the lemma is nondegenerate on the left. To see
that it is nondegenerate on the right, interchange in the above
argument $C$ and $C^{*}$, $i$ and $2-i$, ${\bh}$ and
${\bh}_{\e\text{c}}$ and use the right-hand nondegeneracy of
\eqref{p2}.
\end{proof}

\begin{lemma} Assume that $a\in \bd^{\e 1}\lbe(S,C\e)\cap p^{\e
m}\e \bh^{\e 1}\lbe(S_{\e\rm{fl}},C\e)$ is orthogonal to $\bd^{\e
1}\lbe(S,C^{*}\e)_{p^{\e m}}$ under the pairing of Proposition 4.1.
Then $a\in p^{\e m}\bd^{\e 1}\lbe(S,C\e)$.
\end{lemma}
\begin{proof} The proof is formally the same as the proof of
\cite{dem1}, Lemma 4.5, using Lemma 4.2 above in place of [op.cit.],
Proposition 4.2.
\end{proof}

{\it We can now complete the proof of Proposition 4.1}. By Lemmas
4.3 and 4.4, the pairing of Proposition 4.1 induces a nondegenerate
(and therefore perfect) pairing of finite groups $\bd^{\e
1}\lbe(S,C\e)(p)\times \bd^{\e 1}\lbe(S,C^{*}\e)(p)\ra\bq/\e\bz$.
See \cite{dem1}, proof of Corollary 4.6, p.19.
\end{proof}

Let $G$ be a reductive group scheme over $S$. Then there exist {\it
dual abelian cohomology groups with compact support}
$$
H_{\rm{ab,\e c}}^{\le i}(S_{\e\rm{\acute{e}t}},G^{*})={\bh}^{\e
i}_{\e\text{c}}\big(S_{\e\rm{\acute{e}t}},Z(G\e)^{*}\ra
Z\big(\Gtil\e\big)^{\be *}\e\big)
$$
which fit into an exact sequence
\begin{equation}\label{cs*}
\dots\ra\! H_{\rm{ab,\e c}}^{\le
i}(S_{\e\rm{\acute{e}t}},G^{*})\!\ra\! H_{\rm{ab}}^{\le
i}(S_{\e\rm{\acute{e}t}},G^{*})\!\ra\!\displaystyle\prod_{v\le\in\le\Sigma}H_{\rm{ab}}^{\le
i}(K_{v},G^{*})\!\ra\! H_{\rm{ab,\e c}}^{\le
i+1}(S_{\e\rm{\acute{e}t}},G^{*})\!\ra\dots.
\end{equation}
These groups have the following property (cf. \cite{adt},
Proposition III.0.4(c), p.221, and Remark III.0.6, p.223). Let $V$
be a nonempty open subscheme of $S$. Then there exists an exact
sequence
\begin{equation}\label{cov}
\dots\ra\! H_{\rm{ab,\e c}}^{\le
i}(V_{\e\rm{\acute{e}t}},G^{*})\!\ra\! H_{\rm{ab,\e c}}^{\le
i}(S_{\e\rm{\acute{e}t}},G^{*})\!\ra\!\displaystyle\prod_{v\in
S\e\setminus V}H_{\rm{ab}}^{\le i}(\ove,G^{*})\!\ra\dots.
\end{equation}
Further, by applying an analog of \cite{adt}, Proposition
III.0.4(b), p.220, to the short exact sequence of complexes
$$
1\ra\big(G^{\tor*}\be\ra 1\e\big)\ra \big(Z(G\e)^{*}\ra
Z\big(\Gtil\e\big)^{\be *}\big)\ra\big(Z(G^{\der})^{*}\ra
Z\big(\Gtil\e\big)^{\be *}\e\big)\ra 1
$$
and using the quasi-isomorphism $\big(Z(G^{\der})^{*}\be\ra\be
Z\big(\Gtil\e\big)^{\be *}\e\big)\simeq \big(\e 1\ra\mu^{\be
*}\e\big)$, we conclude that the groups $H_{\rm{ab,\e c}}^{\le
i}(S_{\e\rm{\acute{e}t}},G^{*})$ fit into an exact sequence
\begin{equation}\label{c*}
\dots\ra H^{\e i-1}_{\rm{c}}(S_{\e\rm{\acute{e}t}},\mu^{\be *}\e)\ra
H^{\le i+1}_{\rm{c}}(S_{\e\rm{\acute{e}t}},G^{\tor*}\e)\ra
H_{\rm{ab,\e c}}^{\le i}(S_{\e\rm{\acute{e}t}},G^{*})\ra H^{\e
i}_{\rm{c}}(S_{\e\rm{\acute{e}t}},\mu^{\lbe *}\e)\ra\dots.
\end{equation}

\begin{examples}\indent

\begin{enumerate}

\item[(a)] If $G$ is semisimple, i.e., $G^{\tor}=0$, then $H^{\e i}_{\rm{ab,\e c}}
(S_{\e\rm{\acute{e}t}},G^{*})=H^{\e
i}_{\rm{c}}(S_{\e\rm{\acute{e}t}},\mu^{\lbe *})$.
\item[(b)] If $G$ has trivial fundamental group, i.e., $\mu=0$, then
$H^{\e i}_{\rm{ab,\e c}} (S_{\e\rm{\acute{e}t}},G^{*})=H^{\e
i+1}_{\rm{c}}(S_{\e\rm{\acute{e}t}},G^{\tor*}\e)$.

\end{enumerate}

\end{examples}

Now let $1\ra F\ra H \ra G \ra 1$ be a flasque resolution of $G$ and
let $C=(F\ra R)$, where $R=H^{\tor}$. By \cite{ga4}, Proposition
4.2, the given resolution induces isomorphisms $H^{\le
i}_{\rm{ab}}(S_{\e\rm{fl}},G\e)\simeq{\bh}^{\le
i}(S_{\e\rm{fl}},C\e), H^{\le
i}_{\rm{ab}}(S_{\e\rm{\acute{e}t}},G^{*}\e)\simeq{\bh}^{\le
i}(S_{\e\rm{\acute{e}t}},C^{*})$ and $H_{\rm{ab,\e c}}^{\le
i}(S_{\e\rm{\acute{e}t}},G^{*})\simeq{\bh}^{\e
i}_{\e\text{c}}(S_{\e\rm{\acute{e}t}},C^{*})$. In particular, via
the above isomorphisms, the pairing \eqref{p1} defines a pairing
\begin{equation}\label{pab}
H^{\le 1}_{\rm{ab}}(S_{\e\rm{fl}},G\e)\times H_{\rm{ab,\e c}}^{\le
1}(S_{\e\rm{\acute{e}t}},G^{*})\ra \bq/\bz.
\end{equation}
Clearly, a different choice of flasque resolution of $G$ yields
another such pairing which is isomorphic to \eqref{pab}, i.e.,
\eqref{pab} is independent up to isomorphism of the chosen flasque
resolution of $G$.

Similarly, for every $v\le\in\le\Sigma$, there exist isomorphisms
$H^{\le i}_{\rm{ab}}(K_{v,\rm{fl}},G\e)\simeq {\bh}^{\le
i}(K_{v,\rm{fl}},C\e)$ and $H^{\le i}_{\rm{ab}}(K_{v},G^{*}\e)\simeq
{\bh}^{\le i}(K_{v},C^{*}\e)$. Thus, if $D^{1}_{\rm{ab}}\lbe(S,G\e)$
is the group \eqref{defs} and
\begin{equation}\label{d'}
\begin{array}{rcl}
D^{1}_{\rm{ab}}\lbe(S,G^{*}\e)&=&\krn\!\!\left[\e H^{\e
1}_{\rm{ab}}\lbe(S_{\e\rm{\acute{e}t}},G^{*})
\ra\displaystyle\prod_{v\le\in\le\Sigma}H^{\le
1}_{\rm{ab}}(K_{v},G^{*})\right]\\\\
&=&\img\!\left[\e H_{\rm{ab,\e c}}^{\le
1}(S_{\e\rm{\acute{e}t}},G^{*}) \ra H_{\rm{ab}}^{\le
1}(S_{\e\rm{\acute{e}t}},G^{*})\right],
\end{array}
\end{equation}
then there exist isomorphisms
$D^{1}_{\rm{ab}}\lbe(S,G\e)\simeq\bd^{\e 1}\lbe(S,C\e)$ and
$D^{1}_{\rm{ab}}\lbe(S,G^{*}\e)\simeq\bd^{\e 1}\lbe(S,C^{*}\e)$.
Consequently, the following is an immediate corollary of Proposition
4.1.

\begin{proposition} The pairing \eqref{pab} induces a perfect pairing of finite
groups
$$
D^{1}_{\rm{ab}}(S,G\e)\times D^{1}_{\rm{ab}}(S,G^{*})\ra\bq/\e\bz,
$$
where $D^{1}_{\rm{ab}}(S,G\e)$ and $D^{1}_{\rm{ab}}(S,G^{*})$ are
the groups \eqref{defs} and \eqref{d'}, respectively.\qed
\end{proposition}

\begin{examples} \indent

\begin{enumerate}
\item[(a)] If $G$ is semisimple, then the proposition yields a perfect pairing of
finite groups
$$
D^{2}(S,\mu\e)\times D^{1}(S,\mu^{\be *})\ra\bq/\e\bz.
$$
See Example 2.1(a). Compare \cite{adt}, Corollary II.3.4, p.178, and
\cite{ga1}, Lemma 4.7.

\item[(b)] If $G$ has trivial fundamental group, then the
proposition yields a perfect pairing of finite groups
$$
D^{1}(S,G^{\tor})\times D^{2}(S,G^{\tor*})\ra\bq/\e\bz.
$$
See Example 2.1(b). Compare \cite{adt}, Corollary II.4.7, p.192, and
\cite{ga1}, Theorem 5.7 (for $\s M=(1\ra G^{\tor})$).
\end{enumerate}
\end{examples}

\begin{lemma} Let $Y$ be an $S$-group scheme which is
\'etale-locally isomorphic to ${\bz}^{\! r}$ for some $r\geq 1$.
Then the canonical map $H^{\le 2}(S_{\rm{\acute{e}t}},Y\e)\ra H^{\le
2}(K,Y\e)$ is injective.
\end{lemma}
\begin{proof} Let $K_{\be S}$ be the maximal subfield of the separable closure of $K$
which is unramified at all primes of $K$ which correspond to a
(closed) point of $S$. Further, let $S^{\e\prime}$ be the
normalization of $S$ in $K_{S}$ and write
$\g_{\!S}=\text{Gal}(K_{S}/K)$. Then $Y\times_{S}S^{\e\prime}$ is
constant and the proof of \cite{adt}, Lemma II.2.10, p.172, shows
that $H^{\le 2}(S^{\e\prime}_{\rm{\acute{e}t}},Y\e)$ is a quotient
of $H^{\le 1}(S^{\e\prime}_{\rm{\acute{e}t}},(Y\times\bq)/Y\e)$.
Since the latter group is zero by [op.cit.], proof of Proposition
II.2.9, p.171, so is $H^{\le
2}(S^{\e\prime}_{\rm{\acute{e}t}},Y\e)$. Now, since $H^{\le
1}(S^{\e\prime}_{\rm{\acute{e}t}},Y\e)$ is also zero [loc.cit.], the
exact sequence of terms of low degree belonging to the
Hochschild-Serre spectral sequence $H^{p}(\g_{\!S},
H^{q}(S^{\e\prime}_{\rm{\acute{e}t}},Y))\implies
H^{p+q}(S_{\rm{\acute{e}t}},Y)$ yields an isomorphism
$H^{2}(\g_{\!S},Y)\simeq H^{\le 2}(S_{\rm{\acute{e}t}},Y\e)$. On the
other hand, the canonical map $H^{2}(\g_{\!S},Y)\ra H^{2}(K,Y)$ is
injective (see \cite{hs}, p.112, lines 10-15), and the lemma
follows.
\end{proof}

Now define
\begin{equation}\label{sha*}
\Sha^{1}_{\rm{ab}}\lbe(K,G^{*}\e)=\krn\!\!\left[\e H^{\e
1}_{\rm{ab}}\lbe(K,G^{*}\e) \ra\displaystyle\prod_{\text{all
$v$}}H^{\le 1}_{\rm{ab}}(K_{v},G^{*}\e)\right].
\end{equation}

\begin{proposition} The pairing of Proposition 4.6 induces a perfect
pairing of finite groups
$$
\Sha^{1}_{\rm{ab}}(K,G\e)\times
\Sha^{1}_{\rm{ab}}(K,G^{*})\ra\bq/\e\bz,
$$
where $\Sha^{1}_{\rm{ab}}(K,G\e)$ and $\Sha^{1}_{\rm{ab}}(K,G^{*})$
are the groups \eqref{absha} and \eqref{sha*}, respectively.
\end{proposition}
\begin{proof} We will show that there exist a nonempty open subset $W$ of $S$
and canonical isomorphisms $D^{1}_{\rm{ab}}(W,G\e)\simeq
\Sha^{1}_{\rm{ab}}(K,G\e)$ and $D^{1}_{\rm{ab}}(W,G^{*})\simeq
\Sha^{1}_{\rm{ab}}(K,G^{*})$. The proposition will then follow
immediately from Proposition 4.6. Let $1\ra F\ra H \ra G \ra 1$ be a
flasque resolution of $G$ and set $R=H^{\tor}$. Then, for any
nonempty open subset $U$ of $S$, the given resolution induces
isomorphisms $D^{1}_{\rm{ab}}\lbe(U,G\e)\simeq\bd^{\e 1}\lbe(U,C\e)$
and $D^{1}_{\rm{ab}}\lbe(U,G^{*})\simeq\bd^{\e 1}\lbe(U,C^{*}\e)$,
where $C=(F\ra R)$. Similarly, there exist isomorphisms
$\Sha^{1}_{\rm{ab}}(K,G\e)\simeq\Sha^{1}(K,C)$ and
$\Sha^{1}_{\rm{ab}}(K,G^{*})\simeq \Sha^{1}(K,C^{*})$, where
$\Sha^{1}(K,C)$ (respectively, $\Sha^{1}(K,C^{*})$) is the group
\eqref{shac}. Thus, it suffices to find a set $W$ as above and
isomorphisms $\bd^{\e 1}\lbe(W,C\e)\simeq \Sha^{1}(K,C)$ and
$\bd^{\e 1}\lbe(W,C^{*}\e)\simeq \Sha^{1}(K,C^{*})$. There exists a
nonempty open subset $U$ of $S$ such that $H^{\le
1}(V_{\rm{\acute{e}t}},R\e)=0$ for every open subset $V$ of $U$
(indeed, $R$ is a finite product of tori of the form
$R_{S^{\e\prime}\be/\lbe S}(\bg_{m,S^{\prime}}\lbe)$ where each
$S^{\e\prime}$ is finite and \'etale over $S$ and contains a
nonempty open subscheme with trivial Picard group). It follows that
the canonical map $\bh^{\e 1}(V_{\e\rm{fl}},C\e)\ra H^{\e
2}\lbe(V_{\rm{\acute{e}t}},F\e)$ is injective (cf. \cite{ga4},
Proposition 4.2). Thus there exists an injection $\bd^{\e
1}\lbe(V,C\e)\hookrightarrow D^{2}(V,F)$. Further, the canonical map
$D^{2}(V,F\e)\ra\Sha^{2}(K,F)$ is injective (see the proof of
Theorem 3.12) and $\Sha^{1}(K,C)\ra \Sha^{2}(K,F)$ is an isomorphism
(cf. Proposition 3.1). Now the commutative diagram
$$
\xymatrix{\bd^{\e 1}\lbe(V,C\e)\ar[d]\ar@{^{(}->}[r]&D^{2}(V,F)\ar@{^{(}->}[d]\\
\Sha^{1}(K,C)\ar[r]^(.48){\simeq}&\Sha^{2}(K,F)}
$$
shows that the canonical map $\bd^{\e 1}\lbe(V,C\e)\ra
\Sha^{1}(K,C)$ is injective. This map is shown to be surjective in
\cite{dem1}, proof of Theorem 5.7, p.22, lines 11-12. Thus, for
every $V\subset U$, the canonical map $\bd^{\e 1}\lbe(V,C\e)\ra
\Sha^{1}(K,C)$ is, in fact, an isomorphism. On the other hand, Lemma
4.8 shows that the map $H^{\le 2}(S_{\rm{\acute{e}t}},R^{*}\e)\ra
H^{\le 2}(K,R^{*}\e)$ is injective and \cite{ga1}, proof of Lemma
6.2, shows that the canonical map $H^{\le
1}(S_{\rm{\acute{e}t}},F^{*}\e)\ra H^{\le 1}(K,F^{*}\e)$ is an
isomorphism. In particular, $H^{\le
1}(S_{\rm{\acute{e}t}},R^{*}\e)\simeq H^{\le 1}(K,R^{*}\e)=0$. Thus
there exists an exact commutative diagram
\[
\xymatrix{1\ar[r]& H^{\le
1}(S_{\rm{\acute{e}t}},F^{*}\e)\ar[d]^{\simeq}\ar[r] & \bh^{\e
1}(S_{\rm{\acute{e}t}},C^{*})\ar[d]
\ar[r] & H^{2}(S_{\rm{\acute{e}t}},R^{*}\e)\ar@{^{(}->}[d]\\
1\ar[r] & H^{\le 1}(K,F^{*}\e)\ar[r] & \bh^{\e 1}(K,C^{*})\ar[r] &
H^{2}(K,R^{*}\e)}
\]
(cf. \cite{ga4}, Proposition 4.2), which shows that there exists an
injection $\bh^{\e 1}(S_{\rm{\acute{e}t}},C^{*})\hookrightarrow
\bh^{\e 1}(K,C^{*})$. In particular, $\bd^{\e 1}(V,C^{*})$ may be
identified with a subgroup of $\bh^{\e 1}(K,C^{*})$. Now the same
arguments used in \cite{dem1}, proofs of Lemma 5.6 and Theorem 5.7,
show that there exist a nonempty open subset $U^{*}$ of $S$ and
isomorphisms $\bd^{\e
1}(V,C^{*})\overset{\!\sim}\ra\Sha^{1}(K,C^{*})$ for every nonempty
open subset $V$ of $\,U^{*}$. Further, these isomorphisms are
compatible with respect to inclusions $V^{\e\prime}\subset V\subset
U^{*}$, in the sense that the diagram
\begin{equation}\label{comp}
\xymatrix{\bd^{\e 1}(V^{\e\prime},C^{*})\ar[d]
\ar[r]^(.47){\sim} & \Sha^{1}(K,C^{*}),\\
\bd^{\e 1}(V,C^{*})\ar[ur]^(.47){\sim} &}
\end{equation}
(whose vertical map is induced by the canonical map $\bh^{\e
1}_{\e\text{c}}(V^{\prime}_{\rm{\acute{e}t}},C^{*}\e)\ra\bh^{\e
1}_{\e\text{c}}(V_{\rm{\acute{e}t}},C^{*}\e)$) commutes. We conclude
that, if $W=U\cap U^{*}$, then there exist isomorphisms $\bd^{\e
1}(V,C)\simeq\Sha^{1}(K,C)$ and $\bd^{\e
1}(V,C^{*})\simeq\Sha^{1}(K,C^{*})$ for every nonempty open subset
$V$ of $W$, as desired.
\end{proof}

\begin{examples}\indent

\begin{enumerate}

\item[(a)] If $G$ is semisimple, the pairing of the proposition is a
pairing
$$
\Sha^{2}(K,\mu)\times \Sha^{1}(K,\mu^{\be *})\ra\bq/\e\bz
$$
which is isomorphic to the natural one, i.e., that induced by the
pairing $\mu\times\mu^{\be *}\ra\bg_{m,S}$ (the so-called
Poitou-Tate pairing of \cite{adt}, Theorem I.4.10(a), p.57, and
\cite{ga1}, Theorem 1.1). See Example 2.1(a).

\item[(b)] If $G$ has trivial fundamental group, then the pairing of the
proposition is isomorphic to the natural pairing
$$
\Sha^{1}(K,G^{\tor})\times \Sha^{2}(K,G^{\tor*})\ra\bq/\e\bz
$$
for the $K$-torus $G^{\tor}\!\times_{S}\be\spec K$. See Example
2.1(b).
\end{enumerate}
\end{examples}

\smallskip

As seen in the proof of Proposition 4.9, the canonical map $\bh^{\e
1}(S_{\rm{\acute{e}t}},C^{*})\ra\bh^{\le 1}(K,C^{*})$ is injective.
Thus, by \cite{ga4}, Proposition 4.2, the canonical map $H^{\le
1}_{\rm{ab}}(S_{\rm{\acute{e}t}}, G^{*})\hookrightarrow H^{\le
1}_{\rm{ab}}(K, G^{*})$ is injective as well. Its image
\begin{equation}\label{abnr}
H^{\le 1}_{\rm{ab,nr}}(K, G^{*})=\img\!\lbe\left[\e H^{\le
1}_{\rm{ab}}(S_{\rm{\acute{e}t}}, G^{*})\ra H^{\le 1}_{\rm{ab}}(K,
G^{*})\e\right]
\end{equation}
is called {\it the subgroup of ($S$-)unramified classes of $H^{\le
1}_{\rm{ab}}(K, G^{*})$}. The following statement, which will be
used in the next Section, is immediate from \eqref{d'}.
\begin{proposition} The canonical map $H^{\le
1}_{\rm{ab}}(S_{\rm{\acute{e}t}}, G^{*})\ra H^{\le 1}_{\rm{ab}}(K,
G^{*})$ induces an isomorphism
$$
D^{1}_{\rm{ab}}\lbe(S,G^{*})\simeq H^{\le 1}_{\rm{ab,nr}}(K,
G^{*})\cap\krn\!\!\left[H^{\le 1}_{\rm{ab}}(K,
G^{*})\ra\displaystyle\prod_{v\le\in\le\Sigma}H^{\le
1}_{\rm{ab}}(K_{v}, G^{*})\right].\qed
$$
\end{proposition}

We now recall the exact sequence \eqref{secuab}:
$$
1\ra C_{\rm{ab}}(G\e)\ra
D^{1}_{\rm{ab}}\lbe(S,G\e)\ra\Sha^{1}_{\rm{ab}}\lbe(K,G\e) \ra 1.
$$
By Propositions 4.6 and 4.9, the dual of the preceding exact
sequence of finite abelian groups is an exact sequence
$$
1\ra \Sha^{1}_{\rm{ab}}\lbe(K,G^{*})\ra
D^{1}_{\rm{ab}}\lbe(S,G^{*})\ra C_{\rm{ab}}(G\e)^{D}\ra 1.
$$
Thus the following holds.

\begin{theorem} The pairings of Propositions 4.6 and 4.9 induce a perfect pairing
of finite groups
$$
C_{\rm{ab}}(G\e)\times
D^{1}_{\rm{ab}}\lbe(S,G^{*})/\Sha^{1}_{\rm{ab}}\lbe(K,G^{*})\ra\bq/\bz.
$$
In other words, the exact annihilator of $C_{\rm{ab}}(G\e)\subset
D^{1}_{\rm{ab}}\lbe(S,G\e)$ under the pairing of Proposition 4.6 is
the group $\Sha^{1}_{\rm{ab}}\lbe(K,G^{*})\subset
D^{1}_{\rm{ab}}\lbe(S,G^{*})$.\qed
\end{theorem}

By Examples 4.7 and 4.10, the following statements are immediate
consequences of the theorem.
\begin{corollary} Let $G$ be a semisimple $S$-group scheme with
fundamental group $\mu$. Then there exists a perfect pairing of
finite groups
$$
C_{\rm{ab}}(G\e)\times D^{1}\lbe(S,\mu^{\lbe
*})/\Sha^{1}\lbe(K,\mu^{\lbe *})\ra\bq/\bz.\qed
$$
\end{corollary}

\begin{corollary} Let $T$ be an $S$-torus. Then there exists a perfect pairing of
finite groups
$$
C(T)\times D^{2}\lbe(S,T^{*})/\Sha^{2}\lbe(K,T^{*})\ra\bq/\bz.\qed
$$
\end{corollary}

\smallskip

\begin{remark} If $T$ is an $S$-torus such that
$\Sha^{1}_{S}(K,T)=0$, then $\Sha^{1}(K,T)\subset \Sha^{1}_{S}(K,T)$
is zero as well, whence $\Sha^{2}\lbe(K,T^{*})\simeq
\Sha^{1}(K,T)^{D}=0$. On the other hand,
$H^{1}(S_{\rm{\acute{e}t}},T)\ra\prod_{\e
v\le\in\le\Sigma}H^{1}(K_{v},T)$ is the zero map (see diagram
\eqref{abdiag2}), whence its dual $\prod_{\e
v\le\in\le\Sigma}H^{1}(K_{v},T^{*})\ra H_{\rm{c}}^{\le
2}(S_{\rm{\acute{e}t}},T^{*})$ is also zero. Thus the map
$H_{\rm{c}}^{\le 2}(S_{\rm{\acute{e}t}},T^{*})\ra H^{\le
2}(S_{\rm{\acute{e}t}},T^{*})$ is injective and consequently
$D^{2}\lbe(S,T^{*})\simeq H_{\rm{c}}^{\le
2}(S_{\rm{\acute{e}t}},T^{*})$. Therefore the corollary yields a
perfect pairing $C(T)\times H_{\rm{c}}^{\le
2}(S_{\rm{\acute{e}t}},T^{*})\ra\bq/\bz$ which is isomorphic to the
pairing \eqref{pair}.
\end{remark}

Let
\begin{equation}\label{deltas}
\delta_{S}\colon H_{\rm{ab,\e c}}^{\le
1}(S_{\rm{\acute{e}t}},G^{*})\ra\prod_{\e v\in S_{0}}\be
H_{\rm{ab}}^{\le 1}(\ove,G^{*})
\end{equation}
be the canonical map (see \eqref{cov}). Then we have the following
alternative description of $\e
D^{1}_{\rm{ab}}\lbe(S,G^{*})/\Sha^{1}_{\rm{ab}}\lbe(K,G^{*})\simeq
C_{\rm{ab}}(G\e)^{D}$.

\begin{proposition} There exists a canonical isomorphism of finite
groups
$$
D^{1}_{\rm{ab}}\lbe(S,G^{*})/\Sha^{1}_{\rm{ab}}\lbe(K,G^{*})\simeq\img\delta_{S},
$$
where $\delta_{S}$ is the map \eqref{deltas}.
\end{proposition}
\begin{proof} Recall the set $W$
introduced in the proof of Proposition 4.9. For every nonempty open
subset $V$ of $W$, there exist isomorphisms
$D^{1}_{\rm{ab}}\lbe(V,G^{*})\overset{\sim}\ra\Sha^{1}_{\rm{ab}}\lbe(K,G^{*})$
which are compatible with respect to inclusions $V^{\e\prime}\subset
V\subset W$, i.e., the following diagram commutes
\begin{equation}\label{comp2}
\xymatrix{D^{1}_{\rm{ab}}\lbe(V^{\prime},G^{*})\ar[d]
\ar[r]^(.47){\sim} & \Sha^{1}_{\rm{ab}}\lbe(K,G^{*}).\\
D^{1}_{\rm{ab}}\lbe(V,G^{*})\ar[ur]^(.47){\sim} &}
\end{equation}
Here the vertical arrow is induced by $\bh^{\e 1}_{\e\rm{ab,\e
c}}(V^{\prime}_{\rm{\acute{e}t}},G^{*})\ra\bh^{\e 1}_{\e\rm{ab,\e
c}}(V_{\rm{\acute{e}t}},G^{*})$ (see \eqref{comp}). On the other
hand, for each $V$ the composition
$D^{1}_{\rm{ab}}\lbe(V,G^{*})\overset{\sim}\ra\Sha^{1}_{\rm{ab}}\lbe(K,G^{*})
\hookrightarrow D^{1}_{\rm{ab}}\lbe(S,G^{*})$ is induced by the
canonical map $\bh^{\e 1}_{\e\rm{ab,\e
c}}(V_{\rm{\acute{e}t}},G^{*})\ra\bh^{\e 1}_{\e\rm{ab,\e
c}}(S_{\rm{\acute{e}t}},G^{*})$, and these maps fit into a
commutative diagram
\begin{equation}\label{comp3}
\xymatrix{D^{1}_{\rm{ab}}\lbe(V^{\prime},G^{*})\,\ar[d]
\ar@{^{(}->}[r] & D^{1}_{\rm{ab}}\lbe(S,G^{*}).\\
D^{1}_{\rm{ab}}\lbe(V,G^{*})\,\ar@{^{(}->}[ur] &}
\end{equation}
We will now compute the cokernel of the map
$D^{1}_{\rm{ab}}\lbe(V,G^{*})\hookrightarrow
D^{1}_{\rm{ab}}\lbe(S,G^{*})$ for any nonempty open subset $V$ of
$W$. Since $D^{1}_{\rm{ab}}\lbe(S,G^{*})=\img\e[\e H_{\rm{ab,\e
c}}^{\le 1}(S_{\rm{\acute{e}t}},G^{*})\ra H_{\rm{ab}}^{\le
1}(S_{\rm{\acute{e}t}},G^{*})]$ by \eqref{d'}, \eqref{cs*} induces
an isomorphism
\begin{equation}\label{cock}
\cok\!\!\left[\,\,\displaystyle\prod_{v\le\in\le\Sigma}H_{\rm{ab}}^{\le
0}(K_{v},G^{*})\overset{\beta_{\be S}}\longrightarrow H_{\rm{ab,\e
c}}^{\le 1}(S_{\rm{\acute{e}t}},G^{*})\right]\simeq
D^{1}_{\rm{ab}}\lbe(S,G^{*}),
\end{equation}
and similarly over $V$. Now consider the pairs of maps
\begin{equation}\label{pair1}
\displaystyle\prod_{v\notin V}H_{\rm{ab}}^{\le
0}(K_{v},G^{*})\overset{\alpha_{_{V,S}}}\twoheadrightarrow
\displaystyle\prod_{v\le\in\le\Sigma}H_{\rm{ab}}^{\le
0}(K_{v},G^{*})\overset{\beta_{_{\be S}}}\longrightarrow
H_{\rm{ab,\e c}}^{\le 1}(S_{\rm{\acute{e}t}},G^{*})
\end{equation}
and
\begin{equation}\label{pair2}
\displaystyle\prod_{v\notin V}H_{\rm{ab}}^{\le
0}(K_{v},G^{*})\overset{\beta_{_{\lbe V}}}\longrightarrow
H_{\rm{ab,\e c}}^{\le
1}(V_{\rm{\acute{e}t}},G^{*})\overset{\gamma_{_{
V\!,S}}}\longrightarrow H_{\rm{ab,\e c}}^{\le
1}(S_{\rm{\acute{e}t}},G^{*}),
\end{equation}
which satisfy
$\beta_{S}\circ\alpha_{V,S}=\gamma_{V\!,S}\circ\beta_{V}$. Then
\eqref{pair1} and \eqref{cock} show that
$$
\cok(\gamma_{_{ V\!,S}}\circ\beta_{_{V}})=\cok(\e\beta_{_{\lbe
S}}\circ\alpha_{_{V,S}})=\cok \beta_{_{\lbe S}}\simeq
D^{1}_{\rm{ab}}\lbe(S,G^{*}).
$$
Now \eqref{cock} over $V$, the kernel-cokernel exact sequence of the
pair of maps \eqref{pair2} (see \cite{adt}, Proposition I.0.24,
p.16), and the injectivity of $D^{1}_{\rm{ab}}\lbe(V,G^{*})\ra
D^{1}_{\rm{ab}}\lbe(S,G^{*})$ yield an exact sequence
\begin{equation}\label{kool}
1\ra D^{1}_{\rm{ab}}\lbe(V,G^{*})\ra D^{1}_{\rm{ab}}\lbe(S,G^{*})\ra
\cok\gamma_{_{\e V,S}}\ra 1.
\end{equation}
Finally, we partially order the family of nonempty open subsets $V$
of $S$ by setting $V\leq V^{\prime}\Leftrightarrow V^{\prime}\subset
V$. Then, by the commutativity of \eqref{comp3}, \eqref{kool} is an
exact sequence of inverse systems of finite abelian groups. By
\eqref{comp2},
\begin{equation}\label{inv1}
\displaystyle\varprojlim_{V\subset
S}D^{1}_{\rm{ab}}\lbe(V,G^{*})\simeq
\Sha^{1}_{\rm{ab}}\lbe(K,G^{*}).
\end{equation}
On the other hand, by \eqref{cov}, there exists a canonical
isomorphism
$$
\cok\gamma_{_{\e V,S}}\simeq\img\!\!\left[\e H_{\rm{ab,\e c}}^{\le
1}(S_{\rm{\acute{e}t}},G^{*})\overset{\delta_{_{
V\!,S}}}\longrightarrow \displaystyle\prod_{v\in S\e\setminus
V}H_{\rm{ab}}^{\le 1}(\ove,G^{*})\right],
$$
where the map $\delta_{_{ V\!,S}}$ fits into a commutative diagram
$$
\xymatrix{H_{\rm{ab,\e c}}^{\le
1}(S_{\rm{\acute{e}t}},G^{*})\ar[dr]_(.4){\delta_{_{ V\!,S}}}
\ar[r]^(.47){\delta_{_{\be S}}} & \displaystyle\prod_{v\in
S_{0}}H_{\rm{ab}}^{\le
1}(\ove,G^{*})\ar[d],\\
&\displaystyle\prod_{v\in S\e\setminus V}H_{\rm{ab}}^{\le
1}(\ove,G^{*}).}
$$
Consequently
\begin{equation}\label{inv2}
\displaystyle\varprojlim_{V\subset S}\cok\gamma_{_{\e
V,S}}\simeq\displaystyle\varprojlim_{V\subset S}\delta_{_{
V\!,S}}=\img\e\delta_{S}.
\end{equation}
Thus, by \eqref{inv1} and \eqref{inv2}, the inverse limit of
\eqref{kool} is an exact sequence\footnote{Recall that the inverse
limit functor is exact on the category of finite abelian groups by
\cite{jen}, Proposition 2.3, p.14.}
$$
1\ra \Sha^{1}_{\rm{ab}}\lbe(K,G^{*})\ra
D^{1}_{\rm{ab}}\lbe(S,G^{*})\ra\img\e\delta_{S}\ra 1,
$$
which completes the proof.
\end{proof}

The preceding proposition has interesting consequences, as we will
now see.

Recall the pairing \eqref{pab}:
$$
H^{\le 1}_{\rm{ab}}(S_{\e\rm{fl}},G\e)\times H_{\rm{ab,\e c}}^{\le
1}(S_{\rm{\acute{e}t}},G^{*})\ra \bq/\bz.
$$
When $G$ is semisimple with fundamental group $\mu$, the above
pairing is isomorphic to the canonical pairing
$$
H^{\le 2}(S_{\rm{fl}},\mu)\times H_{\rm{c}}^{\le
1}(S_{\rm{\acute{e}t}},\mu^{\be *})\ra\bq/\bz
$$
of \cite{adt}, III, Corollary 3.2, p.253, and Theorem 8.2, p.290
(see Examples 2.1(a) and 4.5(a)). The latter is a perfect pairing
between the discrete torsion group $H^{\le 2}(S_{\rm{fl}},\mu)$ and
the profinite group $H_{\rm{c}}^{\le 1}(S_{\rm{\acute{e}t}},\mu^{\be
*})$. On the other hand, by Examples 2.1(b) and 4.5(b), when $G=T$
is a torus, \eqref{pab} is isomorphic to the pairing
$$
H^{\le 1}(S_{\rm{\acute{e}t}},T)\times H_{\rm{c}}^{\le
2}(S_{\rm{\acute{e}t}},T^{*})\ra\bq/\bz
$$
of \cite{adt}, Theorem II.4.6(a), p.191, which is also perfect (both
groups are finite by \cite{adt}, Theorem II.4.6(a), p.191). Thus
\eqref{pab} is perfect in two important particular
cases\footnote{However, \eqref{pab} is not perfect in general. More
precisely, its perfectness {\it does not} follow from that of the
two particular cases just mentioned (via the standard five-lemma
argument) because the Pontryagin dual of $H_{\rm{c}}^{\le
3}(S_{\rm{\acute{e}t}},T^{*})$ is not $H^{\le
0}(S_{\rm{\acute{e}t}},T\e)$, but rather its completion relative to
the topology of sugbroups of finite index. See \cite{adt}, Theorem
II.4.6(a), p.191.}. Now, when \eqref{pab} is perfect, the dual of
\eqref{abseq} is an exact sequence
\begin{equation}\label{dabseq}
1\ra \Sha^{1}_{{\rm{ab}}, S}\e(K,G\e)^{D}\ra H_{\rm{ab,\e c}}^{\le
1}(S_{\rm{\acute{e}t}},G^{*})\ra C_{\rm{ab}}(G\e)^{D}\ra 1.
\end{equation}
On the other hand, by Theorem 4.12 and Proposition 4.16, there
exists a canonical isomorphism
$C_{\rm{ab}}(G\e)^{D}\simeq\img\delta_{S}$. Set
\begin{equation}\label{shabc}
\Sha^{1}_{\rm{ab,\e c}}\lbe(S,G^{\lbe *})=\krn\!\left[ H_{\rm{ab,\e
c}}^{\le 1}(S_{\rm{\acute{e}t}},G^{*})\overset{\delta_{_{\be
S}}}\longrightarrow\prod_{\e v\in S_{0}}\be H_{\rm{ab}}^{\le
1}(\ove,G^{*})\right].
\end{equation}
Then \eqref{dabseq} and the isomorphism
$C_{\rm{ab}}(G\e)^{D}\simeq\img\delta_{S}$ yield the following
statement.

\begin{proposition} Assume that the pairing \eqref{pab} is perfect. Then it induces a
perfect pairing
$$
\Sha^{1}_{{\rm{ab}}, S}\e(K,G\e)\times \Sha^{1}_{\rm{ab,\e
c}}\lbe(S,G^{\lbe *})\ra\bq/\bz,
$$
where the left-hand group is the discrete torsion group
\eqref{shabs} and the right-hand group is the profinite group
\eqref{shabc}.\qed
\end{proposition}

When $G$ is semisimple with fundamental group $\mu$, Example 3.9(a)
shows that $\Sha^{1}_{{\rm{ab}},S}\e(K,G\e)$ is isomorphic
to
\begin{equation}\label{sha2} \Sha^{2}_{S}(K,\mu)=\krn\!\!\left[\e
H^{2}(K_{\rm{fl}},\mu)\ra\displaystyle\prod_{v\in S_{\e 0}}H^{\le
2}(K_{v,{\rm{fl}}},\mu)\right].
\end{equation}
On the other hand, since the map
$$
H_{\rm{ab}}^{\le 1}(\ove,G^{*})=H^{\le 1}(\ove,\mu^{*})\ra H^{\le
1}(K_{v},\mu^{\lbe *})
$$
is injective for every $v\in S_{0}$ by \cite{adt}, III, Lemma
1.1(a), p.237, and p.280, we conclude that $\Sha^{1}_{\rm{ab,\e
c}}\lbe(S,G^{*})$ is isomorphic to
\begin{equation}\label{shacmu}
\Sha^{1}_{\rm{c}}\lbe(S,\mu^{\lbe *})=\krn\!\!\left[\e
H_{\rm{c}}^{\le 1}(S_{\rm{\acute{e}t}},\mu^{\lbe
*})\ra\displaystyle\prod_{v\in S_{0}}H^{\le 1}(K_{v},\mu^{\lbe
*})\right].
\end{equation}
Thus the following is an immediate consequence of the proposition.

\begin{corollary} Let $G$ be a semisimple $S$-group scheme with
fundamental group $\mu$. Then the natural pairing $\mu\times
\mu^{\lbe *}\ra\bg_{m,S}$ induces a perfect pairing
$$
\Sha^{2}_{S}(K,\mu)\times\Sha^{1}_{\rm{c}}\lbe(S,\mu^{\lbe
*})\ra\bq/\bz,
$$
where the left-hand group is the discrete torsion group \eqref{sha2}
and the right-hand group is the profinite group \eqref{shacmu}.\qed
\end{corollary}

When $G=T$ is an $S$-torus, we have $\Sha^{1}_{{\rm{ab}},
S}\e(K,G\e)=\Sha^{1}_{S}\e(K,T\e)$ by Example 3.9(b). Note that,
since both $\Sha^{1}(K,T\e)$ and
$\prod_{v\le\in\le\Sigma}H^{1}(K_{v},T)$ are finite,
$\Sha^{1}_{S}\e(K,T\e)$ is finite as well. On the other hand, by
\cite{dem1}, Lemma 3.2, p.11\footnote{The proof given in [op.cit.]
is also valid in the function field case.}, for every $v\in S_{0}$
the canonical map $H_{\rm{ab}}^{\le 1}(\ove,G^{*})=H^{\le
2}(\ove,T^{*})\ra H^{\le 2}(K_{v},T^{*})$ is injective. Thus, by
Examples 2.1(b) and 4.5(b), we conclude that $\Sha^{1}_{\rm{ab,\e
c}}\lbe(S,T^{*})$ equals
$$
\Sha^{2}_{\rm{c}}\lbe(S,T^{*})=\krn\!\!\left[\e H_{\rm{c}}^{\le
2}(S_{\rm{\acute{e}t}},T^{*})\ra\displaystyle\prod_{v\in
S_{0}}H^{\le 2}(K_{v},T^{*})\right].
$$
We have already noted that $H_{\rm{c}}^{\le
2}(S_{\rm{\acute{e}t}},T^{*})$ is finite, so
$\Sha^{2}_{\rm{c}}\lbe(S,T^{*})$ is finite as well. Thus the
following statement is also an immediate consequence of Proposition
4.17.

\begin{corollary} The natural pairing $T\times
T^{*}\ra\bg_{m,S}$ induces a perfect pairing of finite groups
$$
\Sha^{1}_{S}\e(K,T\e)\times\Sha^{2}_{\rm{c}}\lbe(S,T^{*})\ra\bq/\bz.
\qed
$$
\end{corollary}

\begin{remark} The dualities of the above two corollaries should not be
confused with those contained in \cite{adt}, Theorem I.4.20(a),
p.65. For example, the groups denoted $\Sha^{i}_{S}\e(K,M\e)$ in
[op.cit.], p.56, are {\it not} the same as the groups so denoted
here when $M=\mu$ or $\mu^{\lbe *}$.
\end{remark}

\section{Brauer groups and class groups}

Let $K$ and $S$ be as in the Introduction and let $G$ be a reductive
group scheme over $S$. In this Section we use results from
\cite{bvh} to relate the dual of $C_{\rm{ab}}(G)$ to the algebraic
Brauer group of $G_{\be K}=G\times_{S}\spec K$.

If $X$ is any smooth and geometrically integral $K$-variety, we will
write $\xb$ for $X\times_{\spec K}\spec \kb$.

There exists a canonical complex of abelian groups $\br K\ra\br\e
G_{\be K}\ra\br\,\gb_{\be K}$ and we define
$$
\bra G_{\be K}=\krn\!\be\left[\e\br\e G_{\be K}\ra\br\,\gb_{\be
K}\,\right]/\e\img[\e\br K\ra\br\e G_{\be K}\e].
$$
Further, set
\begin{equation}\label{bg}
\Bha(G_{\be K}\e)=\krn\!\be\left[\,\bra G_{\be
K}\ra\displaystyle\prod_{\text{all $v$}}\bra G_{\be K_{v}}\right].
\end{equation}

For any smooth and geometrically integral $K$-variety $X$, let
${\rm{UPic}}\big(\e\xb\e\big)$ be the complex of $\g$-modules
defined in \cite{bvh}, \S2.1. We note that many of the proofs in
\cite{bvh} are in fact independent of the characteristic of $K$,
provided attention is restricted to reductive groups in [op.cit.],
\S\S 3 and 5, and certain references to the literature in [loc.cit.]
are replaced by references to \cite{san}. Those which are not are
either irrelevant to the matters discussed in this Section, or else
there exist published alternative characteristic-free proofs of
these results (see \cite{bvh}, Remarks 2.14 and 4.15).

There exist a canonical divisor map
$\kb\e[X]^{*}/\e\kc\ra\iv\big(\e\xb\e\big)$ and a canonical
quasi-isomorphism
\begin{equation}\label{upic}
{\rm{UPic}}\big(\e\xb\e\big)[1]\simeq\big(\kb\e[X]^{*}/\e\kc\ra\iv\big(\e\xb\e\big)
\big),
\end{equation}
where $\kb[X]^{*}/\kc$ is placed in degree $-1$ and
$\iv\big(\e\xb\e\big)$ in degree 0 (see \cite{bvh}, Corollary 2.5
and Remark 2.6). Now let $1\ra F\ra H\ra G\ra 1$ be a flasque
resolution of $G$. Then its generic fiber $1\ra F_{\lbe K}\ra
H_{\lbe K}\ra G_{\be K}\ra 1$ is a flasque resolution of $G_{\be K}$
(see \cite{cts2}, Proposition 1.4, p.158), and the following holds.

\begin{proposition} Assume that $G_{\be K}$ admits a smooth
$K$-compactification. Then the flasque resolution $1\ra F_{\lbe
K}\ra H_{\lbe K}\ra G_{\be K} \ra 1$ defines a quasi-isomorphism of
complexes of $\g$-modules
$$
{\rm{UPic}}\big(\e\gb_{\be K}\,\big)[1]\simeq\left(R_{\lbe K}^{*}\ra
F_{\lbe K}^{*}\e\right),
$$
where $R_{\lbe K}=H_{\be K}^{\tor}$.
\end{proposition}
\begin{proof} By \eqref{upic}, ${\rm{UPic}}\big(\e\gb_{\be K}\,\big)[1]$ is
quasi-isomorphic to $\big(\kb\e[G]^{*}/\kc\ra\iv\big(\e\gb_{\be
K}\e\big)\big)$. Now let $X$ be a smooth $K$-compactification of
$G_{\be K}$. Then, by \cite{ct}, Proposition B.2(iii) and Remark
B.2.1(2), pp.130-131, there exist quasi-isomorphisms
$$
\big(\kb\e[G]^{*}/\e\kc\ra\iv\big(\e\gb_{\be
K}\e\big)\big)\simeq\big( \iv_{\e\xb\,\setminus\,\gb_{\lbe
K}}\big(\e\xb\e\big)\ra\pic\big(\e\xb\e\big)\big) \simeq(R_{\lbe
K}^{*}\ra F_{\lbe K}^{*}\e).
$$
This completes the proof.
\end{proof}

\begin{corollary} Under the hypotheses of the above proposition, there exists an
isomorphism $\bra G_{\be K}\simeq H^{\le 1}_{\rm{ab}}(K, G^{*})$.
\end{corollary}
\begin{proof} There exists an isomorphism $\bra G_{\be K}\simeq \bh^{\e 1}\be\big(K,
{\rm{UPic}}\big(\e\gb_{\lbe K}\e\big)[1]\big)$ by \cite{bvh},
Corollary 2.20(ii). On the other hand, the proposition and
\cite{ga4}, Proposition 4.2, yield isomorphisms
$$
\bh^{\e 1}\be\big(K, {\rm{UPic}}\big(\e\gb_{\lbe
K}\e\big)[1]\big)\simeq \bh^{\e 1}(K,R^{*}\ra F^{*}\e)\simeq H^{\e
1}_{\rm{ab}}(K, G^{*}),
$$
whence the result follows.
\end{proof}

Similarly, for each prime $v$ of $K$, there exist isomorphisms $\bra
G_{\lbe K_{v}}\simeq H^{\le 1}_{\rm{ab}}(K_{v}, G^{*})$ which are
compatible with the isomorphism of the preceding corollary. Thus the
chosen flasque resolution of $G$ induces an isomorphism
\begin{equation}\label{bsha}
\Bha(G_{\be K}\e)\simeq \Sha^{1}_{\rm{ab}}\lbe(K,G^{*}\e).
\end{equation}
Therefore the following statement is immediate from Proposition 4.9.

\begin{proposition} Assume that $G_{\be K}$ admits a smooth
$K$-compactification. Then there exists a perfect pairing of finite
groups
$$
\Sha^{1}_{\rm{ab}}(K,G\e)\times\Bha(G_{\be K}\e)\ra\bq/\e\bz,
$$
where $\Sha^{1}_{\rm{ab}}(K,G\e)$ and $\Bha(G_{\be K}\e)$ are the
groups \eqref{absha} and \eqref{bg}, respectively. \qed
\end{proposition}

\begin{remark} We stress the fact that the definition of the above pairing depends
on the choice of a flasque resolution of $G$ (note, however, that a
different choice of flasque resolution leads to an isomorphic
pairing). Now, the proposition and \cite{ga3}, Corollary 5.10,
establish the existence of a pairing
$\Sha^{1}\be(K,G\e)\times\Bha(G_{\be K}\e)\ra\bq/\e\bz$ which
induces a bijection $\Sha^{1}\be(K,G\e)\simeq\Bha(G_{\be K}\e)^{D}$
over any global field $K$. The existence of such a bijection in the
number field case was established in \cite{san}, Theorem 8.5, where
the underlying pairing is the Brauer-Manin pairing. This raises the
question of clarifying the relationship of the pairing of the
proposition with the Brauer-Manin pairing, which we hope to address
in a future publication.
\end{remark}

Now recall the subgroup $H^{\le 1}_{\rm{ab,nr}}(K, G^{*})$ of
$H^{\le 1}_{\rm{ab}}(K, G^{*})$ given by \eqref{abnr} and let
$\branr(G_{\be K}\e)$ denote the subgroup of $\bra G_{\be K}$ which
corresponds to $H^{\le 1}_{\rm{ab,nr}}(K, G^{*})$ under the
isomorphism of Corollary 5.2. Set
\begin{equation}\label{branrs}
\branrs(G_{\be K}\e)=\branr(G_{\be K}\e)\e\cap\krn\!\!\left[\,\bra
G_{\be K}\ra\displaystyle\prod_{v\le\in\le\Sigma}\bra G_{\lbe
K_{v}}\right]\subset\bra G_{\be K}.
\end{equation}
Then, by Corollary 5.2 and Proposition 4.11, there exist
isomorphisms
$$
\begin{array}{rcl}
\branrs(G_{\be K}\e)&\simeq& H^{\le 1}_{\rm{ab,nr}}(K,
G^{*})\cap\krn\!\!\left[H^{\le 1}_{\rm{ab}}(K,
G^{*})\ra\displaystyle\prod_{v\le\in\le\Sigma}H^{\le
1}_{\rm{ab}}(K_{v},
G^{*})\right]\\
&\simeq& D^{1}_{\rm{ab}}\lbe(S,G^{*}).
\end{array}
$$
Thus, using \eqref{bsha}, the following statement is an immediate
consequence of Theorem 4.12.

\begin{theorem} Assume that $G_{\be K}$ admits a smooth
$K$-compactification. Then there exists a perfect pairing of finite
groups
$$
C_{\rm{ab}}(G\e)\times \branrs(G_{\be K}\e)/\Bha(G_{\be
K}\e)\ra\bq/\bz,
$$
where the groups $\branrs(G_{\be K}\e)$ and $\Bha(G_{\be K}\e)$ are
given by \eqref{branrs} and \eqref{bg}, respectively.\qed
\end{theorem}

\end{document}